\documentclass[12pt]{article}
\usepackage{amsmath,amssymb,amsthm}
\numberwithin{equation}{section}
\usepackage{units}
\usepackage{color}
\usepackage[T1]{fontenc}
\usepackage[utf8]{inputenc}
\usepackage{authblk}
\usepackage{bm}
\usepackage{graphicx}

\usepackage{datetime}

\usepackage[colorlinks=true,
linkcolor=webgreen,
filecolor=webbrown,
citecolor=webgreen]{hyperref}

\definecolor{webgreen}{rgb}{0,.5,0}
\definecolor{webbrown}{rgb}{.6,0,0}

\usepackage[toc,page]{appendix}

\newtheorem{theorem}{Theorem}
\newtheorem*{thm}{Theorem}
\newtheorem*{lemma}{Lemma}

\newtheorem{proposition}[theorem]{Proposition}
\newtheorem*{remark}{Remark}

\title{Fibonacci Identities via Fibonacci Functions}
\author[]{Kunle Adegoke \\\href{mailto:adegoke00gmail.com}{\tt adegoke00gmail.com}}

\affil{Department of Physics and Engineering Physics, \mbox{Obafemi Awolowo University}, 220005 Ile-Ife, Nigeria}

\begin{document}
\date{}

\maketitle

\begin{abstract}\noindent
We present a differential-calculus-based method which allows one to derive more identities from {\it any} given Fibonacci-Lucas identity containing a finite number of terms and having at least one free index. The method has two {\it independent} components. The first component allows new identities to be obtained directly from an existing identity while the second yields a generalization of the existing identity. The strength of the first component is that no additional information is required about the given original identity. We illustrate the method by providing new generalizations of some well-known identities such as d'Ocagne identity, Candido's identity, Gelin-Ces\`aro identity and Catalan's identity. The method readily extends to a generalized Fibonacci sequence.
\end{abstract}
\noindent 2010 {\it Mathematics Subject Classification}:
Primary 11B39; Secondary 11B37.

\noindent \emph{Keywords: }
Fibonacci number, Lucas number, gibonacci sequence, generalized Fibonacci sequence, Fibonacci function, summation identity, Gelin-Ces\`aro identity, Candido's identity, d'Ocagne identity.

\section{Introduction}
Let $F_j$ and $L_j$ be the $j$th Fibonacci and Lucas numbers, defined for all integers by
\begin{equation}\label{eq.ei9m9n1}
F_j  = \frac{{\alpha ^j  - \beta ^j }}{{\alpha  - \beta }}\,,\quad L_j  = \alpha ^j  + \beta ^j\,,
\end{equation}
where $\alpha=(1+\sqrt 5)/2$, the golden ratio, and $\beta=(1-\sqrt 5)/2=-1/\alpha$. Of course, $\alpha+\beta=1$, $\alpha\beta=-1$, $\alpha-\beta=\sqrt 5$. Let $(G_j)_{j\in\mathbb Z}$ be the gibonacci sequence having the same recurrence relation as the Fibonacci sequence but starting with arbitrary initial values; that is, let
\begin{equation*}
G_j  = G_{j - 1}  + G_{j - 2},\quad (j \ge 2),
\end{equation*}
with $G_0$ and $G_1$ arbitrary numbers (usually integers) not both zero; and
\begin{equation*}
G_{-j} = G_{-(j - 2)} - G_{-(j - 1)}.
\end{equation*} 
If, inspred by~\eqref{eq.ei9m9n1}, we introduce infinite times differentiable, complex-valued Fibonacci and Lucas functions, $f(x)$ and $l(x)$, defined by
\begin{equation}\label{eq.r280nsg}
f(x)=\frac{\alpha^x - \beta^x}{\alpha - \beta},\quad l(x)=\alpha^x + \beta^x,\quad x\in\mathbb R;
\end{equation}
then, clearly,
\begin{equation}
\left. f(x) \right|_{x = j\in\mathbb Z}  = F_j ,\quad\left. l(x) \right|_{x = j\in\mathbb Z}  = L_j;
\end{equation}
and we will show that (see ~\S\ref{sec.ts07vcf} and~\S\ref{sec.additional}):
\begin{equation}\label{eq.wyg8jr6}
\Re\left( {\left. {\frac{d}{{dx}}f(x)} \right|_{x = j\in\mathbb Z} } \right) = \frac{{L_j }}{{\sqrt 5 }}\,\ln \alpha ,\quad\Re\left( {\left. {\frac{d}{{dx}}l(x)} \right|_{x = j\in\mathbb Z} } \right) = F_j \sqrt 5\, \ln \alpha,
\end{equation}
and
\begin{equation}\label{eq.pagcd4g}
\Im\left( {\left. {\frac{d}{dx}f(x)} \right|_{x = j\in\mathbb Z} } \right) =  - \frac{{\pi \beta ^j }}{{\sqrt 5 }},\quad\Im\left( {\left. {\frac{d}{{dx}}l(x)} \right|_{x = j\in\mathbb Z} } \right) = \pi \beta ^j;
\end{equation}
where, here and throughout this paper, $\Re(X)$ or $\Re X$ denotes the real part of $X$ and $\Im(X)$ or $\Im X$ stands for the imaginary part of $X$. Many authors have studied various Fibonacci and Lucas functions in the past; we mention Halsey~\cite{halsey65}, Parker~\cite{parker68}, Spickerman~\cite{spickerman70}, Horadam and Shannon~\cite{horadam88} and Han et. al~\cite{hans12}. The main difference between the approach in this paper and that in previous work by other authors is that the latter focused on seeking real-valued Fibonacci and Lucas functions. It is, precisely, the complex-valued nature of the Fibonacci and Lucas functions defined in~\eqref{eq.r280nsg} and their derivatives that motivated the method developed in this paper.

\bigskip

Our goal is to present a two-component method, based on~\eqref{eq.r280nsg}--\eqref{eq.pagcd4g} and their extensions, which allows the discovery of more identities from any known Fibonacci-Lucas identity or any gibonacci identity having at least one free index; that is an index that is not being summed over.

\bigskip

To illustrate what we mean, consider the identity
\begin{equation}\label{eq.iigew21}
\sum_{j = 0}^{4n + 1} {( - 1)^{j - 1} \binom{4n + 1}jF_{j + k}^4 }  = 25^n \left( {F_{2n + k + 1}^4  - F_{2n + k}^4 } \right),
\end{equation}
derived, among other similar results, by Hoggatt and Bicknell~\cite{hoggatt64b}. This identity has a free index, $k$. Working only with the knowledge of~\eqref{eq.iigew21}, our method (first component) allows us to derive the following presumably new identity:
\begin{equation}\label{eq.f6y5qaq}
\sum_{j = 0}^{4n + 1} {( - 1)^{j - 1} \binom{4n + 1}jF_{j + k}^3 L_{j + k} }  = 25^n \left( {F_{2n + k + 1}^3 L_{2n + k + 1}  - F_{2n + k}^3 L_{2n + k} } \right);
\end{equation}
which, in turn, implies the identity
\begin{equation}\label{eq.jesqs9v}
\sum_{j = 0}^{4n + 1} {( - 1)^{j - 1} \binom{4n + 1}jF_{2j + k}^2 }  = 25^n F_{2(4n + k + 1)}.
\end{equation}
We are not done yet, as~\eqref{eq.jesqs9v} implies
\begin{equation}\label{eq.rtq2eat}
\sum_{j = 0}^{4n + 1} {( - 1)^{j - 1} \binom{4n + 1}jF_{4j + 2k} }  = 25^n L_{2(4n + k + 1)};
\end{equation}
which finally implies
\begin{equation}\label{eq.pk54e82}
\sum_{j = 0}^{4n + 1} {( - 1)^{j - 1} \binom{4n + 1}jL_{4j + 2k} }  = 5^{2n + 1} F_{2(4n + k + 1)}.
\end{equation}
Thus, the four identities~\eqref{eq.f6y5qaq},~\eqref{eq.jesqs9v},~\eqref{eq.rtq2eat} and~\eqref{eq.pk54e82} all follow from a knowledge of~\eqref{eq.iigew21}.

\bigskip

Our method (second component) provides a generalization of~\eqref{eq.iigew21} to
\begin{equation*}\tag{\ref{eq.p7si1ht}}
\sum_{j = 0}^{4n + 1} {( - 1)^{j - 1} \binom{4n + 1}jF_{j + k}^3 G_{j + r}} = 25^n \left( {F_{2n + k + 1}^3 G_{2n + r + 1} - F_{2n + k}^3 G_{2n + r}} \right).
\end{equation*}
A further generalization is derived in Proposition~\ref{prop.ri7x357} on page~\pageref{prop.ri7x357}.

\bigskip

As another example, consider the following well-known identity (see, for example, Hoggatt and Ruggles~\cite[Theorem 4]{hoggatt66}):
\begin{equation}\label{eq.k9bx5vw}
\tan ^{ - 1} \frac{1}{{F_{2k + 1} }} = \tan ^{ - 1} \frac{1}{{F_{2k} }} - \tan ^{ - 1} \frac{1}{{F_{2k + 2} }}.
\end{equation}
Our method (first component) shows that~\eqref{eq.k9bx5vw} implies the following apparently new identity:
\begin{equation*}
\frac{{L_{2k + 1} }}{{F_{2k + 1}^2  + 1}} = \frac{{L_{2k} }}{{F_{2k}^2  + 1}} - \frac{{L_{2k + 2} }}{{F_{2k + 2}^2  + 1}};
\end{equation*}
and the method (second component) yields a generalization:
\begin{equation*}
\frac{{G_{2k + r + 3}  + G_{2k + r + 1} }}{{F_{2k + 1} (F_{2k + 1}^2  + 1)}} - \frac{{G_{r + 1} }}{{F_{2k + 1}^2  + 1}} = \frac{{G_{r + 2} }}{{F_{2k}^2  + 1}} - \frac{{G_r }}{{F_{2k + 2}^2  + 1}}.
\end{equation*}
Yet another example, our method (first component) shows that the following identity of Howard~\cite[Corollary 3.5]{howard}:
\begin{equation}\label{eq.e715lgu}
F_sG_{k + r} + (-1)^{r - 1}F_{s - r}G_k = F_r G_{k + s},
\end{equation}
having three free indices $r$, $s$ and $k$, implies the following identities:
\begin{gather}
L_s G_{k + r}  + ( - 1)^{r - 1} L_{s - r} G_k  = F_r (G_{k + s + 1}  + G_{k + s - 1} ),\label{eq.nh7xbr5}\\
F_s \left( {G_{k + r + 1}  + G_{k + r - 1} } \right) + ( - 1)^r L_{s - r} G_k  = L_r G_{k + s}\label{eq.llqxumw},\\
L_s \left( {G_{k + r + 1}  + G_{k + r - 1} } \right) + ( - 1)^r 5F_{s - r} G_k  = L_r \left( {G_{k + s + 1}  + G_{k + s - 1} } \right).
\end{gather}
The method (second component) provides the following generalization of~\eqref{eq.e715lgu}:
\begin{equation}\label{eq.smzdu8d}
H_s G_{k + r}  + ( - 1)^{r - 1} H_{s - r} G_k  = F_r \left( {G_0 H_{k + s - 1} + G_1 H_{k + s} } \right);
\end{equation}
where here, and throughout this paper, $(H_{j})_{j\in\mathbb Z}$ is a gibonacci sequence with seeds $H_0$ and~$H_1$. Another generalization of~\eqref{eq.e715lgu} is given in Proposition~\ref{prop.o7pee2v} on page~\pageref{prop.o7pee2v}.

\bigskip

In section~\ref{sec.additional} we will apply the method (second component) to provide a generalization of Candido's identity:
\begin{equation*}
2\left( {F_k^4  + F_{k + 1}^4  + F_{k + 2}^4 } \right) = \left( {F_k^2  + F_{k + 1}^2  + F_{k + 2}^2 } \right)^2,
\end{equation*}
to the following:
\begin{equation*}\tag{\ref{eq.y80lcc4}}
\begin{split}
&2\left( {H_rG_k^3  + H_{r + 1} G_{k + 1}^3  + H_{r + 2} G_{k + 2}^3 } \right)\\
&\qquad = \left( {G_k^2  + G_{k + 1}^2  + G_{k + 2}^2 } \right)\left( {H_r G_k  + H_{r + 1} G_{k + 1}  + H_{r + 2} G_{k + 2} } \right),
\end{split}
\end{equation*}
with a further generalization given in Proposition~\ref{prop.wzey6tg} on page~\pageref{prop.wzey6tg}; a particular case of which is
\begin{equation*}
\begin{split}
&6\left( {F_k F_r F_s F_t  + F_{k + 1} F_{r + 1} F_{s + 1} F_{t + 1}  + F_{k + 2} F_{r + 2} F_{s + 2} F_{t + 2} } \right)\\
& = \left( {F_k F_s  + F_{k + 1} F_{s + 1}  + F_{k + 2} F_{s + 2} } \right)\left( {F_r F_t  + F_{r + 1} F_{t + 1}  + F_{r + 2} F_{t + 2} } \right)\\
&\quad + \left( {F_k F_r  + F_{k + 1} F_{r + 1}  + F_{k + 2} F_{r + 2} } \right)\left( {F_s F_t  + F_{s + 1} F_{t + 1}  + F_{s + 2} F_{t + 2} } \right)\\
&\quad + \left( {F_k F_t  + F_{k + 1} F_{t + 1}  + F_{k + 2} F_{t + 2} } \right)\left( {F_r F_s  + F_{r + 1} F_{s + 1}  + F_{r + 2} F_{s + 2} } \right).
\end{split}
\end{equation*}
Our method (second component) extends the d'Ocagne identity:
\begin{equation*}
F_{r + 1}F_k - F_rF_{k + 1}=(-1)^rF_{k - r},
\end{equation*}
to the gibonacci sequence as
\begin{equation*}
G_{r + 1}G_k - G_rG_{k + 1}=(-1)^r\left(G_1G_{k - r} - G_0G_{k - r + 1}\right);
\end{equation*} 
and extends the well-known formula for the sum of the squares of two consecutive Fibonacci numbers, namely,
\begin{equation*}
F_{k + 1}^2 + F_k^2 =F_{2k + 1},
\end{equation*}
to the gibonacci sequence as
\begin{equation*}
G_{k + 1}^2 + G_k^2 =G_0G_{2k} + G_1G_{2k + 1}.
\end{equation*}
We will also establish the following generalization of Catalan's identity:
\begin{equation*}\tag{\ref{eq.ghaetrs}}
F_{k + r} G_{k - r + s}  + F_{k - r} G_{k + r + s}  = 2F_k G_{k + s}  + ( - 1)^{k + r + 1} F_r^2 \left( {G_{s + 1}  + G_{s - 1} } \right),
\end{equation*}
and of the Gelin-Ces\`aro identity:
\begin{equation*}
\begin{split}
&H_{k + r - 2} G_{k - 1} G_{k + 1} G_{k + 2}  + G_{k - 2} H_{k + r - 1} G_{k + 1} G_{k + 2} \\
&\quad + G_{k - 2} G_{k - 1} H_{k + r + 1} G_{k + 2}  + G_{k - 2} G_{k - 1} G_{k + 1} H_{k + r + 2} \\
& = 4H_{k + r} G_k^3  - 2e_G G_0 \left( {H_{r + 2}  + H_r } \right) + 2e_G G_1 \left( {H_{r + 1}  + H_{r - 1} } \right),
\end{split}
\end{equation*}
The method (second component) extends the fundamental identity of Fibonacci and Lucas numbers,
\begin{equation*}
5F_k^2 - L_k^2 = (-1)^{k - 1}4,
\end{equation*}
to the gibonacci sequence as
\begin{equation*}
5G_k^2  - \left( {G_{k + 1}  + G_{k - 1} } \right)^2  = ( - 1)^k 4e_G;
\end{equation*}
and offers an extension of the triple-angle formula of Lucas
\begin{equation*}
F_{3k}=F_{k + 1}^3 + F_k^3 - F_{k - 1}^3,
\end{equation*}
to
\begin{equation*}
G_0^2 G_{3k - 2}  + 2G_0 G_1 G_{3k - 1}  + G_1^2 G_{3k}  = G_{k + 1}^3  + G_k^3  - G_{k - 1}^3.
\end{equation*}
Consider the generalized Fibonacci sequence $(W_j) = \left(W_j(W_0,W_1; p)\right)$ defined, for all integers and arbitrary real numbers $W_0$, $W_1$ and $p\ne 0$, by the recurrence relation
\begin{equation}\label{eq.k54ys0s}
W_j  = pW_{j - 1}  + W_{j - 2}, \quad  j \ge 2,
\end{equation}
with $W_{ - j}  =W_{-j + 2} - pW_{-j + 1}$.

\bigskip

Two important cases of  $(W_j)$ are the special Lucas sequences of the first kind, $(U_j(p))=(W_j(0,1;p))$, and of the second kind, $(V_j(p))=(W_j(2,p;p))$; so that 
\begin{equation}\label{eq.s6z1bcx}
U_0  = 0,\,U_1  = 1,\quad U_j  = pU_{j - 1}  + U_{j - 2}, \quad  j \ge 2,
\end{equation}
and
\begin{equation}\label{eq.zq47g0a}
V_0  = 2,\,V_1  = p,\quad V_j  = pV_{j - 1}  + V_{j - 2}, \quad  j \ge 2,
\end{equation}
with $U_{ - j}  =U_{-j + 2} - pU_{-j + 1}$ and $V_{ - j}  =V_{-j + 2} - pV_{-j + 1}$.

\bigskip 

We will show that the new method also applies to the generalized Fibonacci sequence. For example, the method (first component) shows that the identity~\cite[Equation (3.14)]{horadam65}:
\begin{equation}
U_r W_{k + 1}  + U_{r - 1} W_k  = W_{k + r}
\end{equation}
implies
\begin{equation}
V_r W_{k + 1}  + V_{r - 1} W_k  = W_{k + r + 1}  + W_{k + r - 1}\label{eq.o0dg2ly}.
\end{equation}
The new method presented in this paper provides some illumination on some observations noted by researchers (for example Long~\cite{long86}, Dresel~\cite{dresel93} and Melham~\cite{melham99}). 

\bigskip

The rest of the paper is arranged as follows. In section~\ref{sec.method} we describe the method (first component) and give examples. We cast about for identities to apply the method (first component) in section~\ref{sec.nglclif}. Further justification of the method (first component) is addressed in section~\ref{sec.justification}. A description of the method (second component), with examples including various extensions and generalizations of some known identities, is presented in section~\ref{sec.additional}. Finally, an extension of the method (second component) to the general second order (Horadam) sequence is offered in section~\ref{sec.horadam}.
\section{The method, first component}\label{sec.method}
Delaying further rigorous justification to section \ref{sec.justification}, we present the method (first component) and give examples. 

\bigskip

Here then is how to obtain more identities from any given Fibonacci-Lucas identity having a free index:
\begin{enumerate}
\item \label{step1} Let $k$ be a free index in the known identity. Replace each Fibonacci number, say $F_{h(k,\cdots)}$, with a certain differentiable function of $k$, namely, $f(h(k,\cdots))$, with $k$ now considered a variable; and replace each Lucas number, say $L_{h(k,\cdots)}$, with a certain differentiable function $l(h(k,\cdots))$. The subscript $h$ will be considered a function of several variables; that is variable $k$ and other parameters (if any) indicated by ellipses $\cdots$. The explicit form of $f(h(k,\cdots))$ or $l(h(k,\cdots))$ will not enter into play.
\item\label{step2} By applying the usual rules of calculus, differentiate, with respect to $k$, through the identity obtained in step~\ref{step1}.
\item\label{step3} Simplify the equation obtained in step~\ref{step2} and take the real part of the whole expression/equation, using also the following prescription:
\begin{quote}
\begin{align}
&\Re \left(f(h(k,\cdots))Y\right)=F_{h(k,\cdots)}\Re Y,\label{eq.dvyw1rc}\\
&\Re \left(l(h(k,\cdots))Y\right)=L_{h(k,\cdots)}\Re Y,\label{eq.pdu0l2a}\\
&\Re\frac{\partial }{{\partial k}}f(h(k, \cdots )) = \frac{L_{h(k,\cdots)}}{{\sqrt 5 }}\ln\alpha,\label{eq.wjannba}\\
&\Re\frac{\partial }{{\partial k}}l(h(k, \cdots )) = F_{h(k,\cdots)}\sqrt 5 \ln\alpha\label{eq.wpu2vfi}.
\end{align}
\end{quote}

\end{enumerate}
\begin{remark}
Formally, the method (first component) of obtaining new identities from a known Fibonacci-Lucas identity proceeds in two quick steps:
\begin{enumerate}
\item[i]\label{rule1} Treat the subscripts of Fibonacci and Lucas numbers as variables and differentiate through the given identity, with respect to the free index of interest, using the rules of differential calculus.
\item[ii]\label{rule2} Make the following replacements:
\begin{gather}
\frac{\partial }{{\partial k}}F_{h(k, \cdots )}  \to \frac{{L_{h(k, \cdots )} }}{{\sqrt 5 }}\ln\alpha\frac{\partial }{{\partial k}}h(k, \cdots ),\label{eq.v15rnmx}\\
\frac{\partial }{{\partial k}}L_{h(k, \cdots )}  \to F_{h(k, \cdots )} \sqrt 5\, \ln\alpha\frac{\partial }{{\partial k}}h(k, \cdots )\label{eq.rws86w7},\\
i\to 0;
\end{gather}
where $i=\sqrt{-1}$ is the imaginary unit.
\end{enumerate}
For example, given the double-angle identity:
\begin{equation*}
F_{2k}=L_kF_k,
\end{equation*}
we have, by step~i,
\begin{equation*}
\frac{d}{{dk}}F_{2k}  = \frac{d}{{dk}}\left(L_kF_k\right)=L_k \frac{d}{{dk}}F_k  + F_k \frac{d}{{dk}}L_k;
\end{equation*}
so that, by step~ii, using~\eqref{eq.v15rnmx} and~\eqref{eq.rws86w7},  we get
\begin{equation*}
\frac{{L_{2k} }}{{\sqrt 5 }} \times \ln\alpha\times\frac{d}{{dk}}(2k) = L_k  \times \frac{{L_k }}{{\sqrt 5 }} \times \ln\alpha + F_k  \times F_k \sqrt 5\times \ln\alpha; 
\end{equation*}
and hence,
\begin{equation}
2L_{2k}  = L_k^2  + 5F_k^2.
\end{equation}

\end{remark}
\subsection{Examples}
We illustrate the method (first component) with some examples from familiar identities.
\subsection{Example from a connecting formula between Fibonacci and Lucas numbers}
In this example we show that 
\begin{equation*}
L_k=F_{k + 1} + F_{k - 1}\implies 5F_k=L_{k + 1} + L_{k - 1}.
\end{equation*}
Following step~\ref{step1} we write
\begin{equation}
l(k)=f(k + 1) + f(k - 1)
\end{equation}
and (step~\ref{step2}) differentiate with respect to $k$, obtaining
\begin{equation*}
\frac{d}{{dk}}l(k) = \frac{d}{{dk}}f(k + 1) + \frac{d}{{dk}}f(k - 1).
\end{equation*}
Step~\ref{step3} now gives
\begin{equation*}
\Re\frac{d}{{dk}}l(k) = \Re\frac{d}{{dk}}f(k + 1) + \Re\frac{d}{{dk}}f(k - 1);
\end{equation*}
and by~\eqref{eq.wjannba} and~\eqref{eq.wpu2vfi},
\begin{equation*}
F_k \sqrt 5\,\ln\alpha  = \frac{{L_{k + 1} }}{{\sqrt 5 }}\,\ln\alpha + \frac{{L_{k - 1} }}{{\sqrt 5 }}\,\ln\alpha;
\end{equation*}
that is
\begin{equation*}
5F_k=L_{k + 1} + L_{k - 1}.
\end{equation*}
\subsubsection{Example from the double-angle identity of Fibonacci and Lucas numbers}
In this example we demonstrate that:
\begin{equation}\label{eq.bcj8yuc}
F_{2k}  = L_k F_k  \implies 2L_{2k}  = L_k^2  + 5F_k^2.
\end{equation}
For the identity $F_{2k}  = L_k F_k$, step~\ref{step1} is
\begin{equation*}
f(2k)=l(k)f(k);
\end{equation*}
where $k$ is now considered a variable.

\bigskip

Following step~\ref{step2}, we differentiate with respect to $k$ to obtain
\begin{equation*}
2\frac{d}{{dk}}f(2k) = l(k)\frac{d}{{dk}}f(k) + f(k)\frac{d}{{dk}}l(k).
\end{equation*}
Step~\ref{step3} gives
\begin{equation*}
2\Re\frac{d}{{dk}}f(2k) = L_k\Re\frac{d}{{dk}}f(k) + F_k\Re\frac{d}{{dk}}l(k).
\end{equation*}
Thus, using~\eqref{eq.wjannba} and~\eqref{eq.wpu2vfi}, we have
\begin{equation*}
2\frac{{L_{2k}}}{{\sqrt 5 }} \ln\alpha= L_k\frac{L_k}{{\sqrt 5 }}\ln\alpha + F_k\sqrt 5 F_k\ln\alpha;
\end{equation*}
which, dropping $\ln\alpha$ and multiplying through by $\sqrt 5$, is
\begin{equation*}
2L_{2k}  = L_k^2  + 5F_k^2.
\end{equation*}
The interested reader may wish to verify that
\begin{equation*}
2L_{2k}  = L_k^2  + 5F_k^2\implies F_{2k}  = L_k F_k .
\end{equation*}
\subsubsection{Example from the multiplication formula of Fibonacci and Lucas numbers}
Here we show that the multiplication formula
\begin{equation*}
F_{k + m}+(-1)^mF_{k - m}=L_mF_k
\end{equation*}
implies
\begin{equation}\label{eq.yma04av}
L_{k + m}  + ( - 1)^m L_{k - m}  = L_m L_k
\end{equation}
and
\begin{equation}\label{eq.cgx28m3}
L_{k + m}  - ( - 1)^m L_{k - m}  = 5F_m F_k.
\end{equation}
We write
\begin{equation}\label{eq.tr4f7bs}
f(k + m) + ( - 1)^m f(k - m) = l(m)f(k);
\end{equation}
so that, treating $k$ as the free index of interest gives
\begin{equation*}
\frac{\partial }{{\partial k}}f(k + m) + ( - 1)^m \frac{\partial }{{\partial k}}f(k - m) = l(m)\frac{\partial }{{\partial k}}f(k).
\end{equation*}
Thus,
\begin{equation*}
\Re\frac{\partial }{{\partial k}}f(k + m) + ( - 1)^m \Re\frac{\partial }{{\partial k}}f(k - m) = l(m)\Re\frac{\partial }{{\partial k}}f(k);
\end{equation*}
and hence, by step~\ref{step3},
\begin{equation*}
\frac{{L_{k + m}}}{{\sqrt 5 }}\ln \alpha  + ( - 1)^m \frac{{L_{k - m}}}{{\sqrt 5 }}\ln \alpha  = L_m\frac{L_k}{{\sqrt 5 }}\ln \alpha;
\end{equation*}
from which we get~\eqref{eq.yma04av}.

\bigskip

Taking $m$ as the index of interest and differentiating~\eqref{eq.tr4f7bs} with respect to $m$ yields
\begin{equation*}
\frac{\partial }{{\partial m}}f(k + m) - ( - 1)^m \frac{\partial }{{\partial m}}f(k - m) + ( - 1)^m i\pi f(k - m) = f(k)\frac{\partial }{{\partial m}}l(m),
\end{equation*}
so that
\begin{equation*}
\Re\frac{\partial }{{\partial m}}f(k + m) - ( - 1)^m \Re\frac{\partial }{{\partial m}}f(k - m)= F_k\Re\frac{\partial }{{\partial m}}l(m);
\end{equation*}
and hence
\begin{equation*}
\frac{{L_{k + m}}}{{\sqrt 5 }}\ln \alpha  - ( - 1)^m \frac{{L_{k - m}}}{{\sqrt 5 }}\ln \alpha  = F_kF_m\sqrt 5 \ln \alpha ;
\end{equation*}
from which~\eqref{eq.cgx28m3} follows.

\bigskip

The reader may verify that the remaining multiplication formula can be discovered by differentiating~\eqref{eq.yma04av} with respect to $m$.
\subsubsection{Example from an inverse tangent Fibonacci number identity}
Consider the following identity:
\begin{equation}\label{eq.g1t5uad}
\tan ^{ - 1} \frac{{F_{2m} }}{{F_{2k + 2m - 1} }} = \tan ^{ - 1} \frac{{L_m }}{{L_{2k + m - 1} }} - \tan ^{ - 1} \frac{{L_m }}{{L_{2k + 3m - 1} }},\quad\text{$m$ even},
\end{equation}
which can be derived using the inverse tangent addition formula and basic Fibonacci-Lucas identities.

\bigskip

We now demonstrate that~\eqref{eq.g1t5uad} implies
\begin{equation}\label{eq.nmx3acj}
\frac{1}{5}\frac{{F_{2m} L_{2k + 2m - 1} }}{{F_{2k + 2m - 1}^2  + F_{2m}^2 }} = \frac{{L_m F_{2k + m - 1} }}{{L_{2k + m - 1}^2  + L_{m}^2 }} - \frac{{L_m F_{2k + 3m - 1} }}{{L_{2k + 3m - 1}^2  + L_m^2 }},\quad\text{$m$ even}.
\end{equation}
We treat $k$ as the free index of interest. Step~\ref{step1} gives
\begin{equation}
\tan ^{ - 1} \frac{{f(2m)}}{{f(2k + 2m - 1)}} = \tan ^{ - 1} \frac{{l(m)}}{{l(2k + m - 1)}} - \tan ^{ - 1} \frac{{l(m)}}{{l(2k + 3m - 1)}};
\end{equation}
so that step~\ref{step2} yields
\begin{equation}
\begin{split}
&\frac{{2f(2m)}}{{f(2k + 2m - 1)^2  + f(2m)^2 }}\frac{\partial}{{\partial k}}f(2k + 2m - 1)\\
&\qquad = \frac{{2l(m)}}{{l(2k + m - 1)^2  + l(m)^2 }}\frac{\partial}{{\partial k}}l(2k + m - 1)\\
&\qquad\quad - \frac{{2l(m)}}{{l(2k + 3m - 1)^2  + l(m)^2 }}\frac{\partial}{{\partial k}}l(2k + 3m - 1),
\end{split}
\end{equation}
whence taking real part and replacing the derivatives using~\eqref{eq.wjannba} and~\eqref{eq.wpu2vfi} gives~\eqref{eq.nmx3acj}.

By treating $m$ as the free index, the interested reader can verify, using our method, that~\eqref{eq.g1t5uad} also implies
\begin{equation}
\frac{2}{5}\frac{{F_{2k - 1} }}{{F_{2k + 2m - 1}^2  + F_{2m}^2 }} =  - \frac{{F_{2k - 1} }}{{L_{2k + m - 1}^2  + L_m^2 }} + \frac{{F_{2k + 3m - 1} L_m  + F_{2k + 2m - 1} }}{{L_{2k + 3m - 1}^2  + L_m^2 }},\quad\text{$m$ even}.
\end{equation}
\subsection{Extension to a generalized Fibonacci sequence}\label{sec.h7yipve}
We now describe how the method (first component) for obtaining new identities from existing ones works for the generalized Fibonacci sequence $\left(W_j(W_0,W_1; p)\right)$ whose terms are given in~\eqref{eq.k54ys0s}. The scheme is the following.
\begin{enumerate}
\item \label{step1g} Let $k$ be a free index in the known identity. Replace each generalized Fibonacci number, say $W_{h(k,\cdots)}$, with a certain differentiable function of $k$, namely, $w(h(k,\cdots))$, with $k$ now considered a variable.
\item\label{step2g} By applying the usual rules of calculus, differentiate, with respect to $k$, through the identity obtained in step~\ref{step1g}.
\item\label{step3g} Simplify the equation obtained in step~\ref{step2g} and take the real part, using also the following prescription:
\begin{quote}
\begin{align}
&\Re \left(w(h(k,\cdots))Y\right)=W_{h(k,\cdots)}\Re Y,\label{eq.ri3mwyt}\\
&\Re\frac{\partial }{{\partial k}}w(h(k, \cdots )) = \frac{W_{h(k + 1,\cdots)} + W_{h(k - 1,\cdots)}}{\Delta}\ln\tau\label{eq.q9r3ffo};
\end{align}
where $\tau=(p + \Delta)/2$ and $\Delta=\sqrt {p^2 + 4} $.
\end{quote}

\end{enumerate}
Note that, on account of~\eqref{eq.cajhusn} and~\eqref{eq.c70545l}, for the special Lucas sequences,~\eqref{eq.ri3mwyt} and~\eqref{eq.q9r3ffo} reduce to 
\begin{align}
&\Re \left(u(h(k,\cdots))Y\right)=U_{h(k,\cdots)}\Re Y,\label{eq.uv52rhd}\\
&\Re\frac{\partial }{{\partial k}}u(h(k, \cdots )) = \frac{V_{h(k,\cdots)} }{\Delta }\ln\tau\label{eq.tnsd6nm}
\end{align}
and
\begin{align}
&\Re \left(v(h(k,\cdots))Y\right)=V_{h(k,\cdots)}\Re Y,\label{eq.w148eak}\\
&\Re\frac{\partial }{{\partial k}}v(h(k, \cdots )) = U_{h(k,\cdots)}\Delta \ln\tau\label{eq.gwl4lx2};
\end{align}
of which the Fibonacci and Lucas relations~\eqref{eq.dvyw1rc}--\eqref{eq.wpu2vfi} are particular cases.

\bigskip

For the gibonacci sequence,~\eqref{eq.ri3mwyt} and~\eqref{eq.q9r3ffo} reduce to
\begin{align}
&\Re \left(g(h(k,\cdots))Y\right)=G_{h(k,\cdots)}\Re Y,\label{eq.h54si8d}\\
&\Re\frac{\partial }{{\partial k}}g(h(k, \cdots )) = \frac{G_{h(k + 1,\cdots)} + G_{h(k - 1,\cdots)}}{\sqrt 5}\ln\alpha\label{eq.jp0f8zj}.
\end{align}
\subsection{More examples}
We give further examples involving the gibonacci sequence and the generalized Fibonacci sequence.
\subsubsection{Examples from an identity of Howard}
Consider the following identity, derived by Howard~\cite[Corollary 3.5]{howard}:
\begin{equation*}\tag{\ref{eq.e715lgu}}
F_sG_{k + r} + (-1)^{r - 1}F_{s - r}G_k = F_r G_{k + s},
\end{equation*}
Identity~\eqref{eq.e715lgu} has three free indices $r$, $s$ and $k$.

\bigskip

We write
\begin{equation}\label{eq.zgdv7h6}
f(s)g(k + r) + (-1)^{r - 1}f(s-r)g(k)=f(r)g(k+s).
\end{equation}
Treating $s$ as the index of interest and differentiating~\eqref{eq.zgdv7h6} with respect to $s$ gives
\begin{equation}\label{eq.vig7d8s}
g(k + r)\frac{d}{{ds}}f(s) + ( - 1)^{r - 1} g(k)\frac{\partial }{{\partial s}}f(s - r) = f(r)\frac{\partial }{{\partial s}}g(k + s);
\end{equation}
so that, taking the real part, we get
\begin{equation*}
G_{k + r}\Re\frac{d}{{ds}}f(s) + ( - 1)^{r - 1} G_k\Re\frac{\partial }{{\partial s}}f(s - r) = F_r\Re\frac{\partial }{{\partial s}}g(k + s).
\end{equation*}
We now use~\eqref{eq.wjannba} to replace the derivatives on the left hand side and~\eqref{eq.jp0f8zj} to replace the derivative on the right hand side, obtaining
\begin{equation*}\tag{\ref{eq.nh7xbr5}}
L_s G_{k + r}  + ( - 1)^{r - 1} L_{s - r} G_k  = F_r (G_{k + s + 1}  + G_{k + s - 1} ).
\end{equation*}
On the other hand, treating $r$ as the index of interest and differentiating~\eqref{eq.zgdv7h6} with respect to $r$ yields
\begin{equation}\label{eq.dzcinsk}
\begin{split}
&f(s)\frac{\partial }{{\partial r}}g(k + r) + ( - 1)^{r - 1} i\pi f(s - r)g(k) - ( - 1)^{r - 1} g(k)\frac{\partial }{{\partial r}}f(s - r)\\
&\qquad = g(k + s)\frac{d}{{dr}}f(r);
\end{split}
\end{equation}
so that, taking real part,
\begin{equation*}
\begin{split}
&F_s\Re\frac{\partial }{{\partial r}}g(k + r) - ( - 1)^{r - 1} G_k\Re\frac{\partial }{{\partial r}}f(s - r) = G_{k + s}\Re\frac{d}{{dr}}f(r).
\end{split}
\end{equation*}
Use of~\eqref{eq.jp0f8zj} and~\eqref{eq.wjannba} finally gives
\begin{equation*}\tag{\ref{eq.llqxumw}}
F_s \left( {G_{k + r + 1}  + G_{k + r - 1} } \right) + ( - 1)^r L_{s - r} G_k  = L_r G_{k + s}.
\end{equation*}
The interested reader is invited to discover, by differentiating with respect to $s$, that~\eqref{eq.llqxumw} implies
\begin{equation}
L_s \left( {G_{k + r + 1}  + G_{k + r - 1} } \right) + ( - 1)^r 5F_{s - r} G_k  = L_r \left( {G_{k + s + 1}  + G_{k + s - 1} } \right);
\end{equation}
and that differentiating~\eqref{eq.e715lgu} with respect to $k$ does not produce a new result.

\subsubsection{Example from a general recurrence relation}
Consider the following identity of Horadam~\cite[Equation (3.14), $q=-1$]{horadam65}:
\begin{equation}
U_r W_{k + 1}  + U_{r - 1} W_k  = W_{k + r}.
\end{equation}
We write
\begin{equation*}
u(r) w(k + 1)  + u(r - 1)w(k)  = w(k + r );
\end{equation*}
and differentiate with respect to $r$, obtaining
\begin{equation*}
\frac{d}{{dr}}u(r)\times w(k + 1) + \frac{d}{{dr}}u(r - 1)\times w(k) = \frac{\partial }{{\partial r}}w(k + r);
\end{equation*}
so that, taking real part, we find
\begin{equation*}
\Re\frac{d}{{dr}}u(r)\times W_{k + 1} + \Re\frac{d}{{dr}}u(r - 1)\times W_k = \Re\frac{\partial }{{\partial r}}w(k + r);
\end{equation*}
and hence, upon using~\eqref{eq.tnsd6nm} and~\eqref{eq.q9r3ffo} to replace the derivatives, we obtain
\begin{equation*}\tag{\ref{eq.o0dg2ly}}
V_r W_{k + 1}  + V_{r - 1} W_k  = W_{k + r + 1}  + W_{k + r - 1}.
\end{equation*}
In particular,
\begin{gather}
V_r U_{k + 1}  + V_{r - 1} U_k  = V_{k + r}, \\
V_r V_{k + 1}  + V_{r - 1} V_k  = (p^2  + 4)U_{k + r}.
\end{gather}
\subsubsection{Example from a multiplication formula}
Here we will demonstrate that the identity~\cite[Equation (3.16)]{horadam65}:
\begin{equation*}
W_{k + r} + (-1)^rW_{k - r}=V_rW_k
\end{equation*}
implies the identity
\begin{equation}\label{eq.zed4w03}
(W_{k + r + 1}  + W_{k + r - 1} ) - ( - 1)^r (W_{k - r + 1}  + W_{k - r - 1} ) = U_r W_k \Delta ^2.
\end{equation}
We write
\begin{equation*}
w(k + r) + ( - 1)^r w(k - r) = v(r)w(k)
\end{equation*}
and differentiate through with respect to $r$ to obtain
\begin{equation*}
\frac{\partial }{{\partial r}}w(k + r) + ( - 1)^r \pi iw(k - r) - ( - 1)^r \frac{\partial }{{\partial r}}w(k - r) = w(k)\frac{d}{{dr}}v(r);
\end{equation*}
so that
\begin{equation*}
\Re\frac{\partial }{{\partial r}}w(k + r) - ( - 1)^r \Re\frac{\partial }{{\partial r}}w(k - r) = w(k)\Re\frac{d}{{dr}}v(r).
\end{equation*}
Using~\eqref{eq.q9r3ffo} and~\eqref{eq.gwl4lx2}, we get
\begin{equation}
\frac{{W_{k + r + 1}  + W_{k + r - 1} }}{\Delta } - ( - 1)^r \frac{{(W_{k - r + 1}  + W_{k - r - 1} )}}{\Delta } = W_kU_r\Delta;
\end{equation}
and hence~\eqref{eq.zed4w03}.

\bigskip

Identities
\begin{equation}
V_{k + r}  - ( - 1)^r V_{k - r}  = U_k U_r \Delta ^2
\end{equation}
and
\begin{equation}
U_{k + r}  - ( - 1)^r U_{k - r}  = U_r V_k
\end{equation}
are special cases of~\eqref{eq.zed4w03}.
\section{Applications}\label{sec.nglclif}
In this section, we pick various known results from the literature and apply the method (first component) to discover new identities.
\subsection{New identities from an identity of Long}
Long~\cite[Equation (44)]{long90} showed that, for a non-negative integer $n$ and any integers $k$ and~$r$,
\begin{equation}\label{eq.rjsdo3v}
\sum_{j = 0}^n {\binom njF_{r + 2kj} }  = L_k^n F_{r + nk},\quad\text{if $k$ is even}.
\end{equation}
Based on the knowledge of~\eqref{eq.rjsdo3v} alone, we will derive the results stated in the proposition.
\begin{proposition}
If $n$ is a non-negative integer, $k$ is an even integer and $r$ is any integer, then
\begin{gather}
2\sum_{j = 0}^n {j\binom njL_{r + 2kj} }  = 5nL_k^{n - 1} F_{r + nk} F_k  + nL_k^n L_{r + nk} ,\label{eq.h1dlneo}\\
2\sum_{j = 0}^n {j\binom njF_{r + 2kj} }  = nL_k^{n - 1} L_{r + nk} F_k  + nL_k^n F_{r + nk}\label{eq.i7p1ny7}. 
\end{gather}
\end{proposition}
Identity~\eqref{eq.rjsdo3v} contains two free indices $r$ and $k$. Treating $r$ as the index of interest immediately gives the Lucas version of~\eqref{eq.rjsdo3v}, namely,
\begin{equation*}
\sum_{j = 0}^n {\binom njL_{r + 2kj} }  = L_k^n L_{r + nk},\quad\text{if $k$ is even};
\end{equation*}
coming from
\begin{equation*}
\sum_{j = 0}^n {\binom nj\Re\frac{\partial }{{\partial r}}f(r + 2kj)}  = l(k)^n \Re\frac{\partial }{{\partial r}}f(r + nk)
\end{equation*}
and prescription~\eqref{eq.wjannba}.

\bigskip

To derive~\eqref{eq.h1dlneo}, write~\eqref{eq.rjsdo3v} as
\begin{equation*}
\sum_{j = 0}^n {\binom njf(r + 2kj) }  = l(k)^n f(r + nk);
\end{equation*}
treat $k$ as the index of interest and differentiate with respect to $k$ (step~\ref{step2}) to obtain
\begin{equation*}
\sum_{j = 0}^n {2j\binom nj\frac{\partial }{{\partial k}}f(r + 2kj)}  = nl(k)^{n - 1} f(r + nk)\frac{\partial }{{\partial k}}l(k) + nl(k)^n \frac{\partial }{{\partial k}}f(r + nk),
\end{equation*}
and, taking real part,
\begin{equation}\label{eq.q13yt1g}
\sum_{j = 0}^n {2j\binom nj\Re\frac{\partial }{{\partial k}}f(r + 2kj)}  = nL_k^{n - 1} F_{r + nk}\Re\frac{\partial }{{\partial k}}l(k) + nL_k^n \Re\frac{\partial }{{\partial k}}f(r + nk).
\end{equation}
Thus~\eqref{eq.h1dlneo} follows from step~\ref{step3} of section~\ref{sec.method}, after using~\eqref{eq.wjannba} and~\eqref{eq.wpu2vfi} to replace the derivatives in~\eqref{eq.q13yt1g}.

\bigskip

To derive~\eqref{eq.i7p1ny7} treat $r$ as the free index of interest in~\eqref{eq.h1dlneo} and write
\begin{equation*}
2\sum_{j = 0}^n {j\binom nj\frac{\partial }{{\partial r}}l(r + 2kj)}  = 5nL_k^{n - 1} \frac{\partial }{{\partial r}}f(r + nk)f(k) + nL_k^n \frac{\partial }{{\partial r}}l(r + nk).
\end{equation*}
\subsection{New identities arising from an identity of Hoggatt and Bicknell}
Based on~Hoggatt and Bicknell's result~\cite[Identity $2^{'}$]{hoggatt64b}:
\begin{equation*}\tag{\ref{eq.iigew21}}
\sum_{j = 0}^{4n + 1} {( - 1)^{j - 1} \binom{4n + 1}jF_{j + k}^4 }  = 25^n \left( {F_{2n + k + 1}^4  - F_{2n + k}^4 } \right),
\end{equation*}
we wish to derive the four identities~\eqref{eq.f6y5qaq},~\eqref{eq.jesqs9v},~\eqref{eq.rtq2eat} and~\eqref{eq.pk54e82} stated in the Introduction section.

\bigskip

Write~\eqref{eq.iigew21} as
\begin{equation*}
\sum_{j = 0}^{4n + 1} {( - 1)^{j - 1} \binom{4n + 1}jf(j + k)^4 }  = 25^n \left( {f(2n + k + 1)^4  - f(2n + k)^4 } \right);
\end{equation*}
and differentiate through, with respect to $k$, to obtain
\begin{equation}\label{eq.vzdatpd}
\begin{split}
&\sum_{j = 0}^{4n + 1} {( - 1)^{j - 1} \binom{4n + 1}j4f(j + k)^3 \frac{\partial }{{\partial k}}f(j + k)}\\ 
&\qquad = 25^n \left( {4f(2n + k + 1)^3 \frac{\partial }{{\partial k}}f(2n + k + 1) - 4f(2n + k)^3 \frac{\partial }{{\partial k}}f(2n + k)} \right);
\end{split}
\end{equation}
and taking real part:
\begin{equation*}
\begin{split}
&\sum_{j = 0}^{4n + 1} {( - 1)^{j - 1} \binom{4n + 1}j4F_{j + k}^3 \Re\frac{\partial }{{\partial k}}f(j + k)}\\ 
&\qquad = 25^n \left( {4F_{2n + k + 1}^3 \Re\frac{\partial }{{\partial k}}f(2n + k + 1) - 4F_{2n + k}^3 \Re\frac{\partial }{{\partial k}}f(2n + k)} \right).
\end{split}
\end{equation*}
Thus,
\begin{equation*}
\begin{split}
&\sum_{j = 0}^{4n + 1} {( - 1)^{j - 1} \binom{4n + 1}jF_{j + k}^3 \frac{L_{j + k}}{\sqrt 5}}\\ 
&\qquad = 25^n \left( {F_{2n + k + 1}^3 \frac{L_{2n + k + 1}}{\sqrt 5} - F_{2n + k}^3 \frac{L_{2n + k}}{\sqrt 5}} \right);
\end{split}
\end{equation*}
and hence~\eqref{eq.f6y5qaq}. Identities~\eqref{eq.jesqs9v},~\eqref{eq.rtq2eat} and~\eqref{eq.pk54e82} are derived in the same manner; \eqref{eq.jesqs9v} is obtained from~\eqref{eq.f6y5qaq}, etc.
\subsection{New identities from an inverse tangent identity}
\begin{proposition}
If $k$ is any integer, then
\begin{gather}
\frac{{L_{2k + 1} }}{{F_{2k + 1}^2  + 1}} = \frac{{L_{2k} }}{{F_{2k}^2  + 1}} - \frac{{L_{2k + 2} }}{{F_{2k + 2}^2  + 1}}\label{eq.ym6resq},\\
\nonumber\\
\frac{{L_{2k + 1} }}{{L_{2k} L_{2k + 2} }}\frac{{(F_{2k}^2  + 1)(F_{2k + 2}^2  + 1)}}{{(F_{2k + 1}^2  + 1)}} = \frac{{(F_{2k + 2}^2  + 1)}}{{L_{2k + 2} }} - \frac{{(F_{2k}^2  + 1)}}{{L_{2k} }}\label{eq.nzng3sw}.
\end{gather}
\end{proposition}
Recall
\begin{equation*}\tag{\ref{eq.k9bx5vw}}
\tan ^{ - 1} \frac{1}{{F_{2k + 1} }} = \tan ^{ - 1} \frac{1}{{F_{2k} }} - \tan ^{ - 1} \frac{1}{{F_{2k + 2} }}.
\end{equation*}
To derive~\eqref{eq.ym6resq}, write~\eqref{eq.k9bx5vw} as
\begin{equation}\label{eq.hpngjml}
\tan ^{ - 1} \frac{1}{{f(2k + 1) }} = \tan ^{ - 1} \frac{1}{{f(2k) }} - \tan ^{ - 1} \frac{1}{{f(2k + 2) }},
\end{equation}
and differentiate with respect to $k$ to obtain
\begin{equation*}
\begin{split}
&\frac{1}{{f(2k + 1)^2  + 1}}\,\frac{d}{{dk}}f(2k + 1)\\
&\qquad = \frac{1}{{f(2k)^2  + 1}}\,\frac{d}{{dk}}f(2k) - \frac{1}{{f(2k + 2)^2  + 1}}\,\frac{d}{{dk}}f(2k + 2),
\end{split}
\end{equation*}
and, taking real part,
\begin{equation*}
\begin{split}
&\frac{1}{{F_{2k + 1}^2  + 1}}\,\Re\frac{d}{{dk}}f(2k + 1)\\
&\qquad = \frac{1}{{F_{2k}^2  + 1}}\,\Re\frac{d}{{dk}}f(2k) - \frac{1}{{F_{2k + 2}^2  + 1}}\,\Re\frac{d}{{dk}}f(2k + 2),
\end{split}
\end{equation*}
and hence~\eqref{eq.ym6resq}, upon using~\eqref{eq.wjannba}.
Identity~\eqref{eq.nzng3sw} is a rearrangement of~\eqref{eq.ym6resq}.

\bigskip

Simple telescoping of~\eqref{eq.ym6resq} and~\eqref{eq.nzng3sw} produces the results stated in the next proposition.
\begin{proposition}
If $n$ is any integer, then
\begin{gather*}
\sum_{k = 1}^n {\frac{{L_{2k + 1} }}{{F_{2k + 1}^2  + 1}}}  = \frac{3}{2} - \frac{{L_{2(n + 1)} }}{{F_{2(n + 1)}^2  + 1}},\\
\sum_{k = 1}^n {\frac{{L_{2k + 1} }}{{L_{2k} L_{2k + 2} }}\frac{{\left( {F_{2k}^2  + 1} \right)\left( {F_{2k + 2}^2  + 1} \right)}}{{\left( {F_{2k + 1}^2  + 1} \right)}}}  = \frac{{F_{2(n + 1)}^2  + 1}}{{L_{2n + 2} }} - \frac{2}{3};
\end{gather*}
with the limiting case:
\begin{equation*}
\sum_{k = 1}^\infty {\frac{{L_{2k + 1} }}{{F_{2k + 1}^2  + 1}}}  = \frac32.
\end{equation*}
\end{proposition}
\subsection{New identities from an identity of Jennings}
Jennings~\cite[Theorem 2]{jennings98} showed, among results of a similar nature, that
\begin{equation*}
F_k \sum_{j = 0}^n {( - 1)^{(k + 1)(n + j)} \binom{n + j}{2j}L_k^{2j} }  = F_{(2n + 1)k}.
\end{equation*}
Writing
\begin{equation*}
\sum_{j = 0}^n {( - 1)^{(k + 1)(n + j)} \binom{n + j}{2j}l(k)^{2j} }  = \frac{{f((2n + 1)k)}}{{f(k)}}
\end{equation*}
and differentiating with respect to $k$ gives
\begin{equation*}
\begin{split}
&\sum_{j = 0}^n {( - 1)^{(k + 1)(n + j)} (n + j)\pi i\binom{n + j}{2j}l(k)^{2j} }  + \sum_{j = 0}^n {( - 1)^{(k + 1)(n + j)} 2j\binom{n + j}{2j}l(k)^{2j - 1} \frac{d}{{dk}}l(k)} \\
&\qquad = \frac{{2n + 1}}{{f(k)}}\frac{d}{{dk}}f((2n + 1)k) - \frac{{f((2n + 1)k)}}{{f(k)^2 }}\frac{d}{{dk}}f(k),
\end{split}
\end{equation*}
and taking real part,
\begin{equation*}
\begin{split}
&\sum_{j = 0}^n {( - 1)^{(k + 1)(n + j)} 2j\binom{n + j}{2j}L_k^{2j - 1} \Re\frac{d}{{dk}}l(k)} \\
&\qquad = \frac{{2n + 1}}{{F_k}}\Re\frac{d}{{dk}}f((2n + 1)k) - \frac{{F_{(2n + 1)k}}}{{F_k^2 }}\Re\frac{d}{{dk}}f(k),
\end{split}
\end{equation*}
which, by~\eqref{eq.wjannba} and~\eqref{eq.wpu2vfi} gives
\begin{equation*}
\sum_{j = 0}^n {( - 1)^{(k + 1)(n + j)} 2j\binom{n + j}{2j}L_k^{2j - 1} F_k \sqrt 5 } = \frac{{2n + 1}}{F_k}\frac{L_{(2n + 1)k}}{\sqrt 5} - \frac{F_{(2n + 1)k}}{F_k^2}\frac{L_k}{\sqrt 5};
\end{equation*}
and hence the result stated in the next proposition.
\begin{proposition}
For non-negative integers $k$ and $n$, we have
\begin{equation*}
F_k^3 \sum_{j = 0}^n {( - 1)^{(k + 1)(n + j)} j\binom{n + j}{2j}L_k^{2j} }  = \frac{1}{{10}}\left( {(2n + 1)F_{2k} L_{(2n + 1)k}  - F_{(2n + 1)k} L_k^2 } \right).
\end{equation*}
\end{proposition}
We also have the following divisibility property.
\begin{proposition}
If $n$ and $k$ are non-negative integers, then
\begin{equation*}
10 F_k^3\text{ divides }  {(2n + 1)F_{2k} L_{(2n + 1)k}  - F_{(2n + 1)k} L_k^2 } .
\end{equation*}
\end{proposition}
\subsection{New identities from Candido's identity}
Setting $x=G_k$, $y=G_{k + 1}$ in the algebraic identity:
\begin{equation}
2\left( {x^4  + y^4  + (x + y)^4 } \right) = \left( {x^2  + y^2  + (x + y)^2 } \right)^2,
\end{equation}
gives the following generalization of Candido's identity:
\begin{equation}
2\left( {G_k^4  + G_{k + 1}^4  + G_{k + 2}^4 } \right) = \left( {G_k^2  + G_{k + 1}^2  + G_{k + 2}^2 } \right)^2.
\end{equation}
Writing
\begin{equation*}
2(g(k)^4  + g(k + 1)^4  + g(k + 2)^4 ) = \left( {g(k)^2  + g(k + 1)^2  + g(k + 2)^2 } \right)^2,
\end{equation*}
and differentiating with respect to $k$ gives
\begin{equation}\label{eq.eqypotv}
\begin{split}
&2\left( {g(k)^3 \frac{d}{{dk}}g(k) + g(k + 1)^3 \frac{d}{{dk}}g(k + 1) + g(k + 2)^3 \frac{d}{{dk}}g(k + 2)} \right)\\
&\qquad = \left( {g(k)^2  + g(k + 1)^2  + g(k + 2)^2 } \right)\times\\
&\qquad\quad\left( {g(k)\frac{d}{{dk}}g(k) + g(k + 1)\frac{d}{{dk}}g(k + 1) + g(k + 2)\frac{d}{{dk}}g(k + 2)} \right);
\end{split}
\end{equation}
so that applying the prescription~\eqref{eq.h54si8d} and~\eqref{eq.jp0f8zj} yields
\begin{equation*}
\begin{split}
&2\left( {G_k^3 (G_{k + 1}  + G_{k - 1} ) + G_{k + 1}^3 (G_{k + 2}  + G_k ) + G_{k + 2}^3 (G_{k + 3}  + G_{k + 1} )} \right)\\
&\qquad= (G_k^2  + G_{k + 1}^2  + G_{k + 2}^2 )\left( {G_k (G_{k + 1}  + G_{k - 1} ) + } \right.\\
&\qquad\qquad\qquad\left. {G_{k + 1} (G_{k + 2}  + G_k ) + G_{k + 2} (G_{k + 3}  + G_{k + 1} )} \right),
\end{split}
\end{equation*}
which can be arranged as stated in the next proposition.
\begin{proposition}
For every integer $k$,
\begin{equation}
\begin{split}
&\quad\quad G_k^2 \left( {G_{k + 1} (G_{k + 2}  + G_k ) + G_{k + 2} (G_{k + 3}  + G_{k + 1} ) - G_k (G_{k + 1}  + G_{k - 1} )} \right)\\
&\quad + G_{k + 1}^2 \left( {G_k (G_{k + 1}  + G_{k - 1} ) + G_{k + 2} (G_{k + 3}  + G_{k + 1} ) - G_{k + 1} (G_{k + 2}  + G_k )} \right)\\
&\quad + G_{k + 2}^2 \left( {G_k (G_{k + 1}  + G_{k - 1} ) + G_{k + 1} (G_{k + 2}  + G_k ) - G_{k + 2} (G_{k + 3}  + G_{k + 1} )} \right)\\
&\qquad = 0.
\end{split}
\end{equation}
\end{proposition}
In particular,
\begin{gather}
F_k^2 F_{2k + 3}  + F_{k + 1}^2 F_{2k + 2}  = F_{k + 2}^2 F_{2k + 1},\label{eq.metcxcl} \\
L_k^2 F_{2k + 3}  + L_{k + 1}^2 F_{2k + 2}  = L_{k + 2}^2 F_{2k + 1}\label{eq.jr1y16g} .
\end{gather}
Subtraction of~\eqref{eq.metcxcl} from~\eqref{eq.jr1y16g} gives
\begin{equation}
F_{k - 1} F_{k + 1} F_{2k + 3}  + F_k F_{k + 2} F_{2k + 2}  = F_{k + 1} F_{k + 3} F_{2k + 1},
\end{equation}
while their addition yields
\begin{equation}
(F_{k + 1}^2  + F_{k - 1}^2 )F_{2k + 3}  + (F_{k + 2}^2  + F_k^2 )F_{2k + 2}  = (F_{k + 3}^2  + F_{k + 1}^2 )F_{2k + 1}.
\end{equation}
Before closing this section, we bring forth a Candido-type identity of R.~S.~Melham and discover new identities from it. Melham~\cite[Theorem 1]{melham04} has shown that:
\begin{equation}\label{eq.zi7ygj7}
6\left( {\sum_{j = 0}^{2n - 1} {G_{k + j}^2 } } \right)^2  = F_{2n}^2 \left( {G_{k + n - 2}^4  + 4G_{k + n - 1}^4  + 4G_{k + n}^4  + G_{k + n + 1}^4 } \right);
\end{equation}
from which, writing $f(2n)$ for $F_{2n}$, $g(k + n - 2)$ for $G_{k + n - 2}$, etc.~, and differentiating with respect to $k$, we have
\begin{equation}
\begin{split}
&\left( {12\sum_{j = 0}^{2n - 1} {g(k + j)^2 } } \right)\sum_{j = 0}^{2n - 1} {2g(k + j)\frac{\partial }{{\partial k}}g(k + j)} \\
&\qquad= f(2n)^2 \left( {4g(k + n - 2)^3 \frac{\partial }{{\partial k}}g(k + n - 2)} \right. + 16g(k + n - 1)^3 \frac{\partial }{{\partial k}}g(k + n - 1)\\
&\qquad\left. { + 16g(k + n)^3 \frac{\partial }{{\partial k}}g(k + n) + 4g(k + n + 1)^3 \frac{\partial }{{\partial k}}g(k + n + 1)} \right).
\end{split}
\end{equation}
Taking the real part according to the prescription of steps~\ref{step2g} and~\ref{step3g}, using~\eqref{eq.jp0f8zj} to replace the derivatives, we obtain the result stated in the next proposition.
\begin{proposition}
If $n$ is a non-negative integer and $k$ is any integer, then
\begin{equation}
\begin{split}
&6 {\sum_{j = 0}^{2n - 1} {G_{j + k}^2 } } {\sum_{j = 0}^{2n - 1} {G_{j + k} (G_{j + k + 1}  + G_{j + k - 1} )} } \\
&\qquad = F_{2n}^2 \left( {G_{k + n - 2}^3 (G_{k + n - 1}  + G_{k + n - 3} ) + 4G_{k + n - 1}^3 (G_{k + n}  + G_{k + n - 2} )} \right.\\
&\qquad\quad\left. { + 4G_{k + n}^3 (G_{k + n + 1}  + G_{k + n - 1} ) + G_{k + n + 1}^3 (G_{k + n + 2}  + G_{k + n} )} \right).
\end{split}
\end{equation}
\end{proposition}
In particular,
\begin{equation}\label{eq.v8jelu6}
\begin{split}
6\sum_{j = 0}^{2n - 1} {F_{j + k}^2 } \sum_{j = 0}^{2n - 1} {F_{2j + 2k} }  &= F_{2n}^2 \left( {F_{k + n - 2}^2 F_{2(k + n - 2)}  + 4F_{k + n - 1}^2 F_{2(k + n - 1)} } \right. \\
&\qquad\left. { + 4F_{k + n}^2 F_{2(k + n)}  + F_{k + n + 1}^2 F_{2(k + n + 1)} } \right)
\end{split}
\end{equation}
and
\begin{equation}\label{eq.mpnlzyd}
\begin{split}
6\sum_{j = 0}^{2n - 1} {L_{j + k}^2 } \sum_{j = 0}^{2n - 1} {F_{2j + 2k} }  &= F_{2n}^2 \left( {L_{k + n - 2}^2 F_{2(k + n - 2)}  + 4L_{k + n - 1}^2 F_{2(k + n - 1)} } \right. \\
&\qquad\left. { + 4L_{k + n}^2 F_{2(k + n)}  + L_{k + n + 1}^2 F_{2(k + n + 1)} } \right).
\end{split}
\end{equation}
Subtraction of~\eqref{eq.v8jelu6} from~\eqref{eq.mpnlzyd} gives
\begin{equation}
\begin{split}
&6\sum_{j = 0}^{2n - 1} {F_{j + k + 1} F_{j + k - 1} } \sum_{j = 0}^{2n - 1} {F_{2j + 2k} } \\
&\qquad = F_{2n}^2 \left( {F_{k + n - 1} F_{k + n - 3} F_{2(k + n - 2)}  + 4F_{k + n} F_{k + n - 2} F_{2(k + n - 1)} } \right.\\
&\qquad\quad\left. { + 4F_{k + n + 1} F_{k + n - 1} F_{2(k + n)}  + F_{k + n + 2} F_{k + n} F_{2(k + n + 1)} } \right).
\end{split}
\end{equation}

\subsection{New identities from the Gelin-Ces\`aro identity}
The Gelin-Ces\`aro identity
\begin{equation}
F_{k - 2}F_{k - 1}F_{k + 1}F_{k + 2}=F_k^4 - 1
\end{equation}
has the following generalization (Horadam and Shannon~\cite[Identity~(2.5), $q=-1$]{horadam87}):
\begin{equation}
W_{k - 2}W_{k - 1}W_{k + 1}W_{k + 2}=W_k^4 + (-1)^k\gamma_W W_k^2 - \delta_W^2;
\end{equation}
where $\gamma_W=e_W(p^2 - 1)$, $\delta_W=e_Wp$ and $e_W=pW_0W_1 + W_0^2 - W_1^2$.

\bigskip

For the sequence of Lucas numbers, $\gamma_L=0$ and $e_L=5=\delta_L$, so that
\begin{equation}
L_{k - 2}L_{k - 1}L_{k + 1}L_{k + 2}=L_k^4 - 25;
\end{equation}
while for the gibonacci sequence, $\gamma_G=0$, $\delta_G=e_G=G_0G_1 +G_0^2 -G_1^2$ and
\begin{equation}
G_{k - 2}G_{k - 1}G_{k + 1}G_{k + 2}=G_k^4 - e_G^2.
\end{equation}
Writing
\begin{equation*}
w(k - 2)w(k - 1)w(k + 1)w(k + 2) = w(k)^4  + ( - 1)^k \gamma w(k)^2  - \delta_W^2
\end{equation*}
and differentiating with respect to $k$ and making use of~\eqref{eq.ri3mwyt} and~\eqref{eq.q9r3ffo} from section~\ref{sec.h7yipve} yields the result stated in the next proposition.
\begin{proposition}
For every integer $k$,
\begin{equation}\label{eq.oahojxv}
\begin{split}
&(W_{k - 1}  + W_{k - 3} )W_{k - 1} W_{k + 1} W_{k + 2}\\
&\quad  + W_{k - 2} (W_k  + W_{k - 2} )W_{k + 1} W_{k + 2}\\ 
&\quad + W_{k - 2} W_{k - 1} (W_k  + W_{k + 2} )W_{k + 2} \\
&\qquad + W_{k - 2} W_{k - 1} W_{k + 1} (W_{k + 3}  + W_{k + 1} )\\ 
& = 2W_k (W_{k + 1}  + W_{k - 1} )(2W_k^2  + ( - 1)^k \gamma_W ).
\end{split}
\end{equation}
\end{proposition}
In particular,
\begin{equation}
\begin{split}
&(G_{k - 1}  + G_{k - 3} )G_{k - 1} G_{k + 1} G_{k + 2}\\
&\quad  + G_{k - 2} (G_k  + G_{k - 2} )G_{k + 1} G_{k + 2}\\ 
&\quad + G_{k - 2} G_{k - 1} (G_k  + G_{k + 2} )G_{k + 2} \\
&\qquad + G_{k - 2} G_{k - 1} G_{k + 1} (G_{k + 3}  + G_{k + 1} )\\ 
& = 4G_k^3 (G_{k + 1}  + G_{k - 1} );
\end{split}
\end{equation}
with the special cases
\begin{equation}\label{eq.p81g9fx}
F_{k + 1} F_{k + 2} F_{2k - 3}  + F_{k - 1} F_{k - 2} F_{2k + 3}  = 2F_k^3 L_k  = 2F_k^2 F_{2k}
\end{equation}
and
\begin{equation}\label{eq.ujkt1we}
L_{k + 1} L_{k + 2} F_{2k - 3}  + L_{k - 1} L_{k - 2} F_{2k + 3}  = 2L_k^3 F_k  = 2L_k^2 F_{2k};
\end{equation}
where, to arrive at~\eqref{eq.p81g9fx} and~\eqref{eq.ujkt1we}, we used
\begin{equation*}
F_{k + 1} + F_{k - 1}=L_k,\quad L_{k + 1} + L_{k - 1}=5F_k,
\end{equation*}
and~\cite[Identity (16a)]{vajda}
\begin{equation*}
L_mF_n + L_nF_m =2F_{m + n}.
\end{equation*}
Substituting $k + 2$ for $k$ and arranging~\eqref{eq.p81g9fx} and~\eqref{eq.ujkt1we} as
\begin{equation*}
\frac{{F_{2k + 1} }}{{F_k F_{k + 1} }} + \frac{{F_{2k + 7} }}{{F_{k + 3} F_{k + 4} }} = \frac{{2F_{k + 2}^2 F_{2k + 4} }}{{F_{k + 2}^4  - 1}}
\end{equation*}
and
\begin{equation*}
\frac{{F_{2k + 1} }}{{L_k L_{k + 1} }} + \frac{{F_{2k + 7} }}{{L_{k + 3} L_{k + 4} }} = \frac{{2L_{k + 2}^2 F_{2k + 4} }}{{L_{k + 2}^4  - 25}};
\end{equation*}
and the use of the telescoping summation formula
\begin{equation*}
\sum_{k = 1}^n {( - 1)^{k - 1} \left( {f_k  + ( - 1)^{m - 1} f_{k + m} } \right)}  = \sum_{k = 1}^m {( - 1)^{k - 1} f_k }  + ( - 1)^{n - 1} \sum_{k = 1}^m {( - 1)^{k - 1} f_{k + n} }
\end{equation*}
yields the summation identities stated in the next proposition.
\begin{proposition}
If $n$ is a non-negative integer, then
\begin{gather}
\sum_{k = 1}^n {\frac{{( - 1)^{k - 1} F_{k + 2}^2 F_{2k + 4} }}{{F_{k + 2}^4  - 1}}}  = \frac{5}{6} + \frac{{( - 1)^{n - 1} }}{2}\left( {\frac{{F_{2n + 3} }}{{F_{n + 1} F_{n + 2} }} - \frac{{F_{2n + 5} }}{{F_{n + 2} F_{n + 3} }} + \frac{{F_{2n + 7} }}{{F_{n + 3} F_{n + 4} }}} \right)\\
\sum_{k = 1}^n {\frac{{( - 1)^{k - 1} L_{k + 2}^2 F_{2k + 4} }}{{L_{k + 2}^4  - 25}}}  = \frac{5}{{14}} + \frac{{( - 1)^{n - 1} }}{2}\left( {\frac{{F_{2n + 3} }}{{L_{n + 1} L_{n + 2} }} - \frac{{F_{2n + 5} }}{{L_{n + 2} L_{n + 3} }} + \frac{{F_{2n + 7} }}{{L_{n + 3} L_{n + 4} }}} \right);
\end{gather}
with
\begin{gather}
\sum_{k = 1}^\infty {\frac{{( - 1)^{k - 1} F_{k + 2}^2 F_{2k + 4} }}{{F_{k + 2}^4  - 1}}}  = \frac56, \\
\sum_{k = 1}^\infty {\frac{{( - 1)^{k - 1} L_{k + 2}^2 F_{2k + 4} }}{{L_{k + 2}^4  - 25}}}  = \frac 5{14} .
\end{gather}
\end{proposition}
Arranging~\eqref{eq.oahojxv} as
\begin{equation}
\begin{split}
&\frac{{W_{j - 1}  + W_{j - 3} }}{{W_{j - 2} }} + \frac{{W_j  + W_{j - 2} }}{{W_{j - 1} }} + \frac{{W_{j + 2}  + W_j }}{{W_{j + 1} }} + \frac{{W_{j + 3}  + W_{j + 1} }}{{W_{j + 2} }}\\
&\qquad = \frac{{2W_j (W_{j + 1}  + W_{j - 1} )(2W_j^3  + ( - 1)^j \gamma_W )}}{{W_{j - 2} W_{j - 1} W_{j + 1} W_{j + 2} }}
\end{split}
\end{equation}
and summing produces the next result.
\begin{proposition}
If $n$ and $k$ are integers then,
\begin{equation}
\begin{split}
&\sum_{j = 1}^n {\frac{{( - 1)^{j - 1} 2W_{j + k} (W_{j + k + 1}  + W_{j + k - 1} )(2W_{j + k}^2  + ( - 1)^{j + k} \gamma_W )}}{{W_{j + k - 2} W_{j + k - 1} W_{j + k + 1} W_{j + k + 2} }}}\\ 
&\qquad = ( - 1)^{n + 1} \frac{{W_{n + k}  + W_{n + k - 2} }}{{W_{n + k - 1} }} + \frac{{W_k  + W_{k - 2} }}{{W_{k - 1} }}\\
&\qquad\qquad + ( - 1)^{n + 1} \frac{{W_{n + k + 3}  + W_{n + k + 1} }}{{W_{n + k + 2} }} + \frac{{W_{k + 3}  + W_{k + 1} }}{{W_{k + 2} }};
\end{split}
\end{equation}
provided none of the denominators vanishes.
\end{proposition}
\subsection{New identities from a reciprocal series of Fibonacci numbers with subscripts $k2^j$}
In this section we apply our method to discover new results associated with the following identity of Rabinowitz~\cite[Equation (4)]{rabinowitz98}:
\begin{equation}
\sum_{j = 0}^n {\frac{1}{{U_{k2^j } }}}  = \frac{{1 + U_{k - 1} }}{{U_k }} + \frac{{1 - ( - 1)^k }}{{U_{2k} }} - \frac{{U_{k2^n  - 1} }}{{U_{k2^n } }}.
\end{equation}
Writing
\begin{equation}
\sum_{j = 0}^n {\frac{1}{{u(k2^j ) }}}  = \frac{{1 + u(k - 1) }}{{u(k) }} + \frac{{1 - ( - 1)^k }}{{u(2k) }} - \frac{{u(k2^n  - 1) }}{{u(k2^n) }},
\end{equation}
and differentiating with respect to $k$ gives
\begin{equation*}
\begin{split}
\sum_{j = 0}^n {\frac{{ - 2^j}}{{u(k2^j )^2 }}\frac{d}{dk}u(k2^j )}  &= \frac{1}{{u(k)}}\frac{d}{{dk}}u(k - 1) - \frac{{(1 + u(k - 1))}}{{u(k)^2 }}\frac{d}{{dk}}u(k)\\
&\qquad - \frac{{( - 1)^k \pi i}}{{u(2k)}} - \frac{{2(1 - ( - 1)^k )}}{{u(2k)^2 }}\frac{d}{{dk}}u(2k)\\
&\qquad\; - \frac{{2^n }}{{u(k2^n )}}\frac{\partial }{{\partial k}}u(k2^n  - 1) + \frac{{2^n u(k2^n  - 1)}}{{u(k2^n )^2 }}\frac{\partial }{{\partial k}}u(k2^n ).
\end{split}
\end{equation*}
Taking the real part while using~\eqref{eq.uv52rhd}--\eqref{eq.gwl4lx2}, we have the next result.
\begin{proposition}\label{prop.dfkkobu}
If $n$ and $k$ are positive integers, then
\begin{gather*}
\sum_{j = 0}^n {\frac{{2^j V_{k2^j } }}{{U_{k2^j }^2 }}}  = \frac{(-1)^k2 + V_k }{{U_k^2 }} + \frac{{2(1 - ( - 1)^k )V_{2k} }}{{U_{2k}^2 }} - \frac{2^{n + 1}}{U_{k2^n }^2 },\\
\sum_{j = 0}^\infty {\frac{{2^j V_{k2^j } }}{{U_{k2^j }^2 }}}  = \frac{(-1)^k2 + V_k }{{U_k^2 }} + \frac{{2(1 - ( - 1)^k )V_{2k} }}{{U_{2k}^2 }} .
\end{gather*}
\end{proposition}
Note that in arriving at the final form of the first expression in Proposition~\ref{prop.dfkkobu}, we used
\begin{equation*}
U_rV_s - V_rU_s = (-1)^s2U_{r - s}.
\end{equation*}
In particular, we have
\begin{equation*}
\sum_{j = 0}^n {\frac{{2^j V_{2^j } }}{{U_{2^j }^2 }}}  = p + \frac{{2\Delta ^2 }}{{p^2 }} - \frac{{2^{n + 1} }}{{U_{2^n }^2 }}
\end{equation*}
and
\begin{equation*}
\sum_{j = 0}^n {\frac{{2^j V_{2^{j + 1} } }}{{U_{2^{j + 1} }^2 }}}  = \frac{{\Delta ^2 }}{{p^2 }} - \frac{{2^{n + 1} }}{{U_{2^{n + 1} }^2 }};
\end{equation*}
with the special cases
\begin{equation*}
\sum_{j = 0}^n {\frac{{2^j L_{2^j } }}{{F_{2^j }^2 }}}  = 11 - \frac{{2^{n + 1} }}{{F_{2^n }^2 }}
\end{equation*}
and
\begin{equation*}
\sum_{j = 0}^n {\frac{{2^j L_{2^{j + 1} } }}{{F_{2^{j + 1} }^2 }}}  = 5 - \frac{{2^{n + 1} }}{{F_{2^{n + 1} }^2 }}.
\end{equation*}
\section{Justification of the method}\label{sec.justification}
In this section we provide the rationale behind the method that was described in section~\ref{sec.method}. To facilitate the discussion, we need the closed formula for the generalized Fibonacci sequence $(W_j)$.
\subsection{Closed formula for the generalized Fibonacci sequence}
Standard methods for solving difference equations give the closed (Binet) formula of the generalized Fibonacci sequence  $(W_j)$ defined by the recurrence relation~\eqref{eq.k54ys0s}, in the non-degenerate case, $p^2 + 4 > 0$, as
\begin{equation}\label{Binet}
W_j = A\tau^j + B\sigma^j, 
\end{equation}
where
\begin{equation}\label{eq.n2r2300}
A = \frac{{W_1 - W_0 \sigma }}{{\tau  - \sigma }},\quad B = \frac{{W_0\tau  - W_1}}{{\tau  - \sigma }},
\end{equation}
with
\begin{equation}\label{eq.qu7nzq5}
\tau = \frac{p+\sqrt{p^2 + 4}}{2}, \qquad \sigma = \frac{p-\sqrt{p^2 + 4}}{2};
\end{equation}
so that
\begin{equation}\label{eq.dv3pphj}
\tau+\sigma=p,\quad\tau-\sigma=\sqrt{p^2 + 4}=\Delta,\quad\mbox{and }\tau\sigma=-1.
\end{equation}
In particular,
\begin{equation}\label{eq.djr1ak1}
U_j=\frac{\tau^j-\sigma^j}{\tau-\sigma},\quad V_j=\tau^j+\sigma^j.
\end{equation}
Using the Binet formulas, it is readily established that
\begin{equation}\label{eq.t2c6ndd}
U_{-j}=(-1)^{j - 1}U_j,\quad V_{-j}=(-1)^jV_j.
\end{equation}
It is also straightforward to establish the following:
\begin{gather}
U_{j + 1}  + U_{j - 1}  = V_j,\label{eq.cajhusn} \\
U_{j + 1}  - U_{j - 1}  = pU_j, \\
V_{j + 1}  + V_{j - 1}  = U_j \Delta ^2\label{eq.c70545l},
\end{gather}
and
\begin{equation}
V_{j + 1}  - V_{j - 1}  = pV_j.
\end{equation}
As for the gibonacci sequence, we have
\begin{equation}\label{eq.i07nxs0}
G_j  = \frac{{G_1  - G_0 \beta }}{{\sqrt 5 }}\alpha ^j  + \frac{{G_0 \alpha  - G_1 }}{{\sqrt 5 }}\beta ^j .
\end{equation}
\begin{lemma}\label{lem.r4azw23}
For any integer $j$,
\begin{equation}\label{eq.dnu2vbt}
A\tau ^j  -  B\sigma ^j  = \frac{{W_{j + 1}  + W_{j - 1} }}{\Delta }\,.
\end{equation}
\end{lemma}
\begin{proof}
Let $Q= A\tau ^j  -  B\sigma ^j $. Then,
\begin{gather*}
\tau Q =  A\tau ^{j + 1}  +  B\sigma ^{j - 1},\\
\sigma Q =  -A\tau ^{j - 1}  -  B\sigma ^{j + 1}.
\end{gather*}
Thus,
\begin{equation*}
Q \times(\tau  - \sigma )= \left(A\tau ^{j + 1}  + B\sigma ^{j + 1}\right)  + \left(A\tau ^{j - 1}  + B\sigma ^{j - 1}\right);
\end{equation*}
that is
\begin{equation*}
Q\Delta = W_{j + 1}  + W_{j - 1}.
\end{equation*}
\end{proof}
Identity~\eqref{eq.dnu2vbt} is at the heart of the justification of the calculus-based method of obtaining Fibonacci identities.
\subsection{Justification of the method}\label{sec.ts07vcf}
Consider a generalized Fibonacci function $w(x)$ defined by
\begin{equation}\label{eq.l53gn22}
w(x)=A\tau^x + B\sigma^x,\quad x\in\mathbb R,
\end{equation}
where $A$ and $B$ are as defined in~\eqref{eq.n2r2300} and $\tau$ and $\sigma$ are as given in~\eqref{eq.qu7nzq5}.

\bigskip

Corresponding to~\eqref{eq.ei9m9n1}, \eqref{eq.djr1ak1} and~\eqref{eq.i07nxs0}, we have the following special cases of~\eqref{eq.l53gn22}:
\begin{equation*}\tag{\ref{eq.r280nsg}}
f(x)  = \frac{{\alpha ^x  - \beta ^x }}{\sqrt 5}\,,\quad l(x)  = \alpha ^x  + \beta ^x\,,
\end{equation*}
\begin{equation}\label{eq.qyjgp2v}
u(x)=\frac{\tau^x-\sigma^x}{\tau-\sigma},\quad v(x)=\tau^x+\sigma^x,
\end{equation}
and
\begin{equation}\label{eq.r25jazv}
g(x)  = A_G\alpha ^x  + B_G\beta ^x ,
\end{equation}
where
\begin{equation}\label{eq.zumcdp3}
A_G=\frac{{G_1  - G_0 \beta }}{{\sqrt 5 }},\quad B_G=\frac{{G_0 \alpha  - G_1 }}{{\sqrt 5 }}.
\end{equation}
Clearly,
\begin{equation}\label{eq.cfgdre}
w(j) = W_j,\quad j\in\mathbb Z;
\end{equation}
that is
\begin{equation*}
u(j)=U_j,\quad v(j)=V_j,\quad g(j)=G_j,\quad f(j)=F_j,\quad l(j)=L_j,\quad j\in\mathbb Z.
\end{equation*}
\begin{thm}
The following identity holds:
\begin{equation}\label{eq.yeahprn}
\Re\left( \left. {\frac{d}{{dx}}w(x)} \right|_{x = j \in \mathbb Z} \right) = \frac{W_{j + 1} + W_{j - 1}}\Delta\ln\tau,
\end{equation}
where, as usual, $\Re (X)$ denotes the real part of $X$.
\end{thm}
\begin{proof}
We have
\begin{equation*}
\begin{split}
\frac{d}{{dx}}w(x) &= A\tau ^x \ln \tau  + B\sigma ^x \ln \sigma \\
 &= A\tau ^x \ln \tau  - B\sigma ^x \ln \tau  + B\sigma ^x \ln \tau  + B\sigma ^x \ln \sigma\\ 
 &= \left( {A\tau ^x  - B\sigma ^x } \right)\ln \tau  + B\sigma ^x \ln (\tau \sigma );
\end{split}
\end{equation*}
that is,
\begin{equation}\label{eq.sl6udgy}
\frac{d}{{dx}}w(x) = \left( {A\tau ^x  - B\sigma ^x } \right)\ln \tau  + B\sigma ^x \ln ( - 1).
\end{equation}
Evaluating~\eqref{eq.sl6udgy} at $x=j\in\mathbb Z$, we have
\begin{equation}\label{eq.ydfrjqv}
\begin{split}
\left. {\frac{d}{{dx}}w(x)} \right|_{x = j\in\mathbb Z}  &= \left( {A\tau ^j  - B\sigma ^j } \right)\ln \tau  + B\sigma ^j \ln ( - 1)\\
& = \frac{{W_{j + 1}  + W_{j - 1} }}{\Delta }\ln \tau  + B\sigma ^j \ln ( - 1),\text{by~\eqref{eq.dnu2vbt}},
\end{split}
\end{equation}
from which, on taking real part,~\eqref{eq.yeahprn} follows, since $\sigma$, $\Delta$, $B$, $W_{j + 1}$ and $W_{j - 1}$ are real quantities and $\tau$ is a positive number.
\end{proof}
Of course the derivatives given in~\eqref{eq.wyg8jr6} are particular cases of~\eqref{eq.yeahprn} with $\Delta=\sqrt 5$, $F_{j + 1} + F_{j - 1}=L_j$ and $L_{j + 1} + L_{j - 1}=5F_j$. Similarly,~\eqref{eq.q9r3ffo},~\eqref{eq.tnsd6nm},~\eqref{eq.gwl4lx2} and~\eqref{eq.jp0f8zj} are all consequences of~\eqref{eq.yeahprn}.

\bigskip

Thus, given a (generalized) Fibonacci identity having a free index, on account of~\eqref{eq.l53gn22}, \eqref{eq.cfgdre} and~\eqref{eq.yeahprn}, we can replace (generalized) Fibonacci numbers with (generalized) Fibonacci functions, perform differentiation and evaluate at integer values to obtain a new (generalized) Fibonacci identity.
\section{The method, second component}\label{sec.additional}
The imaginary part of~\eqref{eq.ydfrjqv} establishes a connection between powers of $\tau$ and $\sigma$ and the (generalized) Fibonacci numbers; through which new (generalized) Fibonacci identities can be obtained. We have
\begin{equation}\label{eq.hg8x0pw}
\Im\left( \left. {\frac{d}{{dx}}w(x)} \right|_{x = j \in \mathbb Z} \right) = B_W\sigma^j\pi(2m + 1),
\end{equation}
where $m$ is some integer and 
\begin{equation}
B_W = \frac{{W_0\tau  - W_1}}{{\tau  - \sigma }}.
\end{equation}
Specializing to the special Lucas sequences, we have
\begin{equation*}
B_U=-\frac1\Delta,\quad B_V=\frac{2\tau - p}\Delta=\frac{\tau - \sigma}\Delta=1;
\end{equation*}
so that
\begin{equation}
\Im\left( {\left. {\frac{d}{{dx}}u(x)} \right|_{x = j\in\mathbb Z} } \right) =  - \frac{{\pi \sigma ^j }}{\Delta }
\end{equation}
and
\begin{equation}
\Im\left( {\left. {\frac{d}{{dx}}v(x)} \right|_{x = j\in\mathbb Z} } \right) = \pi \sigma ^j.
\end{equation}
For the gibonacci sequence, we have
\begin{equation*}
B_G=\frac{G_0\alpha - G_1}{\sqrt 5};
\end{equation*}
so that
\begin{equation}
\Im\left( {\left. {\frac{d}{{dx}}g(x)} \right|_{x = j\in\mathbb Z} } \right) = \frac{{G_0 \alpha  - G_1 }}{{\sqrt 5 }}\pi \beta ^j.
\end{equation}
For the Fibonacci and Lucas numbers, we have
\begin{equation}
B_F  = \frac{{F_0 \alpha  - F_1 }}{{\alpha  - \beta }} =  - \frac{1}{{\sqrt 5 }},\quad B_L  = \frac{{L_0 \alpha  - L_1 }}{{\alpha  - \beta }} = \frac{{2\alpha  - 1}}{{\sqrt 5 }} = 1;
\end{equation}
so that in view of~\eqref{eq.hg8x0pw}, with the principal value $m=0$, the method described in Section~\ref{sec.method} can now be applied to a Fibonacci-Lucas identity with the prescription in step~\ref{step3} replaced with the following:
\begin{quote}
Simplify the equation obtained in step~\ref{step2} and take the imaginary part of the whole expression/equation, using also the following prescription:
\begin{align}
&\Im\left( f(h(k,\cdots))Y\right)=F_{h(k,\cdots)}\Im Y,\label{eq.mbcae9n}\\
&\Im \left(l(h(k,\cdots))Y\right)=L_{h(k,\cdots)}\Im Y,\label{eq.p4bbs37}\\
&\Im\frac{\partial }{{\partial k}}f(h(k, \cdots )) = -\frac{\beta^{h(k,\cdots)}\pi}{{\sqrt 5 }},\label{eq.bf3c8fj}\\
&\Im\frac{\partial }{{\partial k}}l(h(k, \cdots )) = \beta^{h(k,\cdots)}\pi\label{eq.b00ifyn};
\end{align}
where $Y$ is arbitrary.
\end{quote}
\begin{remark}
Formally, the method (second component) of establishing a connection between the powers of $\beta$ and Fibonacci and Lucas numbers in a given Fibonacci-Lucas identity proceeds in two quick steps:
\begin{enumerate}
\item[i]\label{rule1b} Treat the subscripts of Fibonacci and Lucas numbers as variables and differentiate through the given identity, with respect to the free index of interest, using the rules of differential calculus.
\item[ii]\label{rule2b} Make the following replacements:
\begin{gather}
\frac{\partial }{{\partial k}}F_{h(k, \cdots )}  \to \frac{{-\pi\beta^{h(k, \cdots )} }}{{\sqrt 5 }}\frac{\partial }{{\partial k}}h(k, \cdots ),\label{eq.zbscq1z}\\
\frac{\partial }{{\partial k}}L_{h(k, \cdots )}  \to \pi\beta^{h(k, \cdots )} \frac{\partial }{{\partial k}}h(k, \cdots )\label{eq.clf29qo},\\
i\to 1,\\
\ln\alpha\to 0;
\end{gather}
where $i=\sqrt{-1}$ is the imaginary unit.
\end{enumerate}
For example, given the double-angle identity:
\begin{equation*}
F_{2k}=L_kF_k,
\end{equation*}
we have, by step~i,
\begin{equation*}
\frac{d}{{dk}}F_{2k}  = \frac{d}{{dk}}\left(L_kF_k\right)=L_k \frac{d}{{dk}}F_k  + F_k \frac{d}{{dk}}L_k;
\end{equation*}
so that, by step~ii, using~\eqref{eq.zbscq1z} and~\eqref{eq.clf29qo},  we get
\begin{equation*}
\frac{{-\pi\beta^{2k} }}{{\sqrt 5 }} \times \frac{d}{{dk}}(2k) = L_k  \times \frac{{-\pi\beta^k }}{{\sqrt 5 }} + F_k  \times \pi\beta^k ; 
\end{equation*}
and hence,
\begin{equation*}
2\beta^k  = L_k  - F_k\sqrt 5.
\end{equation*}

\end{remark}
For the special Lucas sequences,~\eqref{eq.mbcae9n}--~\eqref{eq.b00ifyn} read
\begin{quote}
\begin{align}
&\Im\left( u(h(k,\cdots))Y\right)=U_{h(k,\cdots)}\Im Y,\label{eq.m6fbisb}\\
&\Im \left(v(h(k,\cdots))Y\right)=V_{h(k,\cdots)}\Im Y,\\
&\Im\frac{\partial }{{\partial k}}u(h(k, \cdots )) = -\frac{\sigma^{h(k,\cdots)}\pi}{{\Delta }},\label{eq.y6bmkxe}\\
&\Im\frac{\partial }{{\partial k}}v(h(k, \cdots )) = \sigma^{h(k,\cdots)}\pi\label{eq.c5vr4yz};
\end{align}
\end{quote}
while for the gibonacci sequence, we have
\begin{quote}
\begin{align}
&\Im\left( g(h(k,\cdots))Y\right)=G_{h(k,\cdots)}\Im Y,\label{eq.brp6mp3}\\
&\Im\frac{\partial }{{\partial k}}g(h(k, \cdots )) = \frac{(G_0\alpha - G_1)\pi\beta^{h(k,\cdots)}}{{\sqrt 5 }}=B_G\pi\beta^{h(k,\cdots)}\label{eq.k32lr6u}.
\end{align}
\end{quote}
\subsection{Examples}
We now give a couple of examples to illustrate the use of~\eqref{eq.mbcae9n}--\eqref{eq.b00ifyn} and~\eqref{eq.brp6mp3} and~\eqref{eq.k32lr6u} in obtaining new identities from known Fibonacci-Lucas identities.

\bigskip

Note that in the definitions in \eqref{eq.ei9m9n1}, $F_j$ and $L_j$ do not change when $\alpha$ and $\beta$ are interchanged. Thus, $\alpha$ and $\beta$ can be interchanged in any Fibonacci-Lucas identity involving Fibonacci numbers, Lucas numbers, $\alpha$ and $\beta$ and no other irrational numbers. More generally, we have the observation stated in Proposition~\ref{prop.invariant}.
\begin{proposition}\label{prop.invariant}
Any (generalized) Fibonacci identity involving (generalized) Fibonacci numbers as well as $\tau$ and $\sigma$ and no other irrational numbers remains valid under the exchange of $\tau$ and $\sigma$.
\end{proposition}
\begin{proof}
From~\eqref{Binet} and~\eqref{eq.n2r2300}, we have
\begin{equation*}
W_j (\tau ,\sigma ) = \frac{{W_1  - W_0 \sigma }}{{\tau  - \sigma }}\tau ^j  + \frac{{W_0 \tau  - W_1 }}{{\tau  - \sigma }}\sigma ^j.
\end{equation*}
It is straightforward to verify that $W_j (\tau ,\sigma )=W_j (\sigma ,\tau )$; and hence, the proposition.
\end{proof}
The generalizations obtained in this section rest on Proposition~\ref{prop.invariant}.
\subsubsection{Generalizations of the fundamental identity of Fibonacci and Lucas numbers}
Differentiating the Fibonacci-Lucas function form:
\begin{equation*}
5f(k)^2  - l(k)^2  = ( - 1)^{k - 1} 4,
\end{equation*}
of the fundamental identity:
\begin{equation*}
5F_k^2  - L_k^2  = ( - 1)^{k - 1} 4,
\end{equation*}
and applying the prescription of~\eqref{eq.mbcae9n}--\eqref{eq.b00ifyn} yields
\begin{equation}\label{eq.tbxfpjf}
5F_k \beta ^{k + r}  + L_k \beta ^{k + r} \sqrt 5  = ( - 1)^k 2\beta ^r \sqrt 5,
\end{equation}
and also, by Proposition~\ref{prop.invariant},
\begin{equation}\label{eq.dy9a2o8}
5F_k \alpha ^{k + r}  - L_k \alpha ^{k + r} \sqrt 5  = ( - 1)^{k - 1} 2\alpha ^r \sqrt 5.
\end{equation}
Combining~\eqref{eq.tbxfpjf} and~\eqref{eq.dy9a2o8} according to the Binet form~\eqref{eq.i07nxs0} leads to the result stated in Proposition~\ref{prop.evb2hve}.
\begin{proposition}\label{prop.evb2hve}
If $k$ and $r$ are any integers, then
\begin{equation*}
5F_k G_{k + r}  - L_k \left( {G_{k + r + 1}  + G_{k + r - 1} } \right) = ( - 1)^{k - 1} 2\left( {G_{r + 1}  + G_{r - 1} } \right).
\end{equation*}
\end{proposition}
Writing the function form of the identity of Proposition~\ref {prop.evb2hve} as
\begin{equation*}
\begin{split}
&5f(k - r)g(k) - l(k - r)\left( {g(k + 1) + g(k - 1)} \right)\\
&\qquad = ( - 1)^{k - r - 1} 2\left( {g(r + 1) + g(r - 1)} \right),
\end{split}
\end{equation*}
and differentiating with respect to $r$, using again the prescription~\eqref{eq.mbcae9n}--\eqref{eq.b00ifyn} and~\eqref{eq.brp6mp3} and~\eqref{eq.k32lr6u}, we find
\begin{equation}\label{eq.vuld42i}
\begin{split}
&5\beta ^{k - r} G_k  + \beta ^{k - r} \sqrt 5 \left( {G_{k + 1}  + G_{k - 1} } \right)\\
&\qquad = ( - 1)^{k - r} 2\sqrt 5 \left( {G_{r + 1}  + G_{r - 1} } \right) + ( - 1)^{k - r} 2\sqrt 5 \left( {G_0 \beta ^{r - 1}  - G_1 \beta ^r } \right),
\end{split}
\end{equation}
and, on account of Proposition~\ref{prop.invariant}, also
\begin{equation}\label{eq.kog9sik}
\begin{split}
&5\alpha ^{k - r} G_k  - \alpha ^{k - r} \sqrt 5 \left( {G_{k + 1}  + G_{k - 1} } \right)\\
&\qquad = ( - 1)^{k - r + 1} 2\sqrt 5 \left( {G_{r + 1}  + G_{r - 1} } \right) + ( - 1)^{k - r + 1} 2\sqrt 5 \left( {G_0 \alpha ^{r - 1}  - G_1 \alpha ^r } \right).
\end{split}
\end{equation}
Combining~\eqref{eq.vuld42i} and~\eqref{eq.kog9sik} gives the next result.
\begin{proposition}
If $k$, $r$ and $s$ are any integers, then
\begin{equation}\label{eq.csoh1r2}
\begin{split}
&5G_kH_{k + s - r}   - \left( {G_{k + 1}  + G_{k - 1} } \right)\left( {H_{k + s - r + 1}  + H_{k + s - r - 1} } \right)\\
&\qquad = ( - 1)^{k - r + 1} 2\left( {G_{r + 1}  + G_{r - 1} } \right)\left( {H_{s + 1}  + H_{s - 1} } \right)\\
&\quad\qquad + ( - 1)^{k - r} 2\left( {G_0 \left( {H_{r + s}  + H_{r + s - 2} } \right) + G_1 \left( {H_{r + s + 1}  + H_{r + s - 1} } \right)} \right).
\end{split}
\end{equation}
\end{proposition}
Setting $r=0$ in~\eqref{eq.csoh1r2} gives
\begin{equation}
\begin{split}
&5G_k H_{k + s}  - \left( {G_{k + 1}  + G_{k - 1} } \right)\left( {H_{k + s + 1}  + H_{k + s - 1} } \right)\\
&\qquad = ( - 1)^{k - 1} 2\left( {G_1 \left( {H_{s + 1}  + H_{s - 1} } \right) - G_0 \left( {H_{s + 2}  + H_s } \right)} \right),
\end{split}
\end{equation}
which upon using $s=0$ and $(H_j)=(G_j)$ gives
\begin{equation*}
5G_k^2  - \left( {G_{k + 1}  + G_{k - 1} } \right)^2  = ( - 1)^k 4e_G,
\end{equation*}
where, as usual, $e_G  = G_0^2  - G_1^2  + G_0 G_1 $.
\subsubsection{Generalizations of the formula for the sum of squares of two consecutive Fibonacci numbers}
Differentiating the Fibonacci function form:
\begin{equation*}
f(k+1)^2 + f(k)^2=f(2k+1),
\end{equation*}
of the identity
\begin{equation*}
F_{k + 1}^2 + F_k^2=F_{2k + 1},
\end{equation*}
we have
\begin{equation*}
f(k + 1)\frac{d}{{dk}}f(k + 1) + f(k)\frac{d}{{dk}}f(k) = \frac{d}{{dk}}f(2k + 1),
\end{equation*}
and taking imaginary part, by~\eqref{eq.mbcae9n},
\begin{equation}
F_{k + 1} \Im\frac{d}{{dk}}f(k + 1) + F_k \Im\frac{d}{{dk}}f(k) = \Im\frac{d}{{dk}}f(2k + 1),
\end{equation}
and by~\eqref{eq.bf3c8fj},
\begin{equation*}
\beta^{k + 1}F_{k + 1} + \beta^kF_k=\beta^{2k + 1}.
\end{equation*}
Thus
\begin{equation*}
\beta ^{r + k + 1} F_{k + 1}  + \beta ^{r + k} F_k  = \beta ^{2k + r + 1},
\end{equation*}
or
\begin{equation}\label{eq.mk4knz9}
\beta ^{s + 1} F_{k + 1}  + \beta ^s F_k  = \beta ^{k + s + 1},
\end{equation}
where $s$ is an arbitrary integer, and also
\begin{equation}\label{eq.nfgv56u}
\alpha ^{s + 1} F_{k + 1}  + \alpha ^s F_k  = \alpha ^{k + s + 1}.
\end{equation}
Combining~\eqref{eq.mk4knz9} and~\eqref{eq.nfgv56u} according to the Binet formula, we have the next result, equivalent to Vajda~\cite[Identity (8)]{vajda}.
\begin{proposition}
If $k$ and $s$ are integers, then
\begin{equation}\label{eq.r6ox2i4}
F_{k + 1} G_{s + 1}  + F_k G_s  = G_{k + s + 1}.
\end{equation}
\end{proposition}
Writing $k - 1$ for $k$ and setting $s=0$ in~\eqref{eq.r6ox2i4} gives the well-known result:
\begin{equation}\label{eq.jxu8c83}
F_k G_1  + F_{k - 1} G_0  = G_k.
\end{equation}
Differentiating the Fibonacci function form of~\eqref{eq.r6ox2i4}, that is
\begin{equation*}
f(k + 1)g(s + 1) + f(k)g(s) = g(k + s + 1),
\end{equation*}
with respect to $k$ and using~\eqref{eq.bf3c8fj} and~\eqref{eq.k32lr6u} gives
\begin{equation*}
 - \beta ^{k + 1} G_{s + 1}  - \beta ^k G_s  = \beta ^{k + s + 1} \left( {G_0 \alpha  - G_1 } \right),
\end{equation*}
or
\begin{equation}\label{eq.pik4d8i}
\beta ^{k + 1} G_{s + 1}  + \beta ^k G_s  = G_0 \beta ^{k + s}  + G_1 \beta ^{k + s + 1}
\end{equation}
and also
\begin{equation}\label{eq.ke586ae}
\alpha ^{k + 1} G_{s + 1}  + \alpha ^k G_s  = G_0 \alpha ^{k + s}  + G_1 \alpha ^{k + s + 1}.
\end{equation}
Combining~\eqref{eq.pik4d8i} and~\eqref{eq.ke586ae} produces the next result.
\begin{proposition}
If $k$ and $s$ are integers, then
\begin{equation*}
H_{k + 1} G_{s + 1}  + H_k G_s  = G_0 H_{k + s}  + G_1 H_{k + s + 1}.
\end{equation*}
\end{proposition}
In particular,
\begin{equation*}
G_{k + 1}^2  + G_k^2  = G_0 G_{2k}  + G_1 G_{2k + 1}.
\end{equation*}
\subsubsection{Generalizations of the d'Ocagne identity}
Differentiating with respect to $k$, the Fibonacci function form:
\begin{equation*}
f(r + 1)f(k) - f(r)f(k + 1)=(-1)^rf(k-r),
\end{equation*}
of the d'Ocagne identity:
\begin{equation*}
F_{r + 1}F_k - F_rF_{k + 1}=(-1)^rF_{k - r},
\end{equation*}
gives, upon taking the imaginary part while using~\eqref{eq.mbcae9n} and~\eqref{eq.bf3c8fj},
\begin{equation}
\beta ^k F_{r + 1}  - \beta ^{k + 1} F_r  = ( - 1)^r \beta ^{k - r},
\end{equation}
and also
\begin{equation}
\alpha ^k F_{r + 1}  - \alpha ^{k + 1} F_r  = ( - 1)^r \alpha ^{k - r},
\end{equation}
and hence, the result stated in the next proposition.
\begin{proposition}
If $r$ and $k$ are any integers, then
\begin{equation}\label{eq.gyme8yy}
F_{r + 1}G_k - F_rG_{k + 1}=(-1)^rG_{k - r}.
\end{equation}
\end{proposition}
Differentiating the Fibonacci function form of~\eqref{eq.gyme8yy}:
\begin{equation*}
f(r + 1)g(k) - f(r)g(k + 1) = ( - 1)^r g(k - r),
\end{equation*}
with respect to $r$, we find
\begin{equation}
\begin{split}
&g(k)\frac{d}{{dr}}f(r + 1) - g(k + 1)\frac{d}{{dr}}f(r)\\
&\qquad = ( - 1)^r \pi g(k - r)i - ( - 1)^r \frac{d}{{dr}}g(k - r);
\end{split}
\end{equation}
and consequent upon use of~\eqref{eq.bf3c8fj},~\eqref{eq.brp6mp3} and~\eqref{eq.k32lr6u},
\begin{equation*}
\begin{split}
&\beta ^{r + s} G_{k + 1}  - \beta ^{r + s + 1} G_k \\
&\qquad = ( - 1)^r \beta ^s G_{k - r} \sqrt 5  + ( - 1)^r \left( {G_0 \beta ^{k + s - r - 1}  + G_1 \beta ^{k + s - r} } \right)
\end{split}
\end{equation*}
and also
\begin{equation*}
\begin{split}
&\alpha ^{r + s} G_{k + 1}  - \alpha ^{r + s + 1} G_k \\
&\qquad = ( - 1)^{r - 1} \alpha ^s G_{k - r} \sqrt 5  + ( - 1)^r \left( {G_0 \alpha ^{k + s - r - 1}  + G_1 \alpha ^{k + s - r} } \right)
\end{split}
\end{equation*}
and hence, the identity stated in the next proposition.
\begin{proposition}
If $r$, $s$ and $k$ are any integers, then
\begin{equation}\label{eq.d9p08dl}
\begin{split}
&H_{r + s} G_{k + 1}  - H_{r + s + 1} G_k \\
&= ( - 1)^{r - 1} \left( {H_{s + 1}  + H_{s - 1} } \right)G_{k - r}\\
&\qquad  + ( - 1)^r \left( {G_0 H_{k + s - r - 1}  + G_1 H_{k + s - r} } \right).
\end{split}
\end{equation}
\end{proposition}
Writing $r - s$ for $r$ in~\eqref{eq.d9p08dl} and setting $s=0$ gives
\begin{equation*}
\begin{split}
H_r G_{k + 1}  - H_{r + 1} G_k  &= ( - 1)^{r - 1} \left( {2H_1  - H_0 } \right)G_{k - r} \\
&\qquad + ( - 1)^r \left( {G_0 H_{k - r - 1}  + G_1 H_{k - r} } \right);
\end{split}
\end{equation*}
and, in particular,
\begin{equation*}
G_rG_{k + 1} - G_{r + 1}G_k=(-1)^r\left(G_0G_{k - r + 1} - G_1G_{k - r}\right).
\end{equation*} 
\subsubsection{Generalizations of Fibonacci power formulas}
The well-known identity:
\begin{equation}\label{eq.zrwo56t}
G_{k + 1}^2 + G_{k - 2}^2=2\left(G_k^2 + G_{k - 1}^2\right),
\end{equation}
has the gibonacci function form:
\begin{equation*}
g(k + 1)^2 + g(k - 2)^2=2\left(g(k)^2 + g(k - 1)^2\right),
\end{equation*}
which, upon differentiation, gives
\begin{equation*}
\begin{split}
&g(k + 1)\frac{\partial }{{\partial k}}g(k + 1) + g(k - 2)\frac{\partial }{{\partial k}}g(k - 2)\\
&\qquad = 2g(k)\frac{\partial }{{\partial k}}g(k) + 2g(k - 1)\frac{\partial }{{\partial k}}g(k - 1),
\end{split}
\end{equation*}
and by~\eqref{eq.brp6mp3} and~\eqref{eq.k32lr6u}:
\begin{equation*}
\beta ^{k + 1} G_{k + 1}  + \beta ^{k - 2} G_{k - 2}  = 2\left( {\beta ^k G_k  + \beta ^{k - 1} G_{k - 1} } \right);
\end{equation*}
and multiplying through by $\beta^{s - k}$, $s$ an arbitrary integer:
\begin{equation}\label{eq.qf73jtg}
\beta ^{s + 1} G_{k + 1}  + \beta ^{s - 2} G_{k - 2}  = 2\left( {\beta ^s G_k  + \beta ^{s - 1} G_{k - 1} } \right);
\end{equation}
and also
\begin{equation}\label{eq.s6sm0bz}
\alpha ^{s + 1} G_{k + 1}  + \alpha ^{s - 2} G_{k - 2}  = 2\left( {\alpha ^s G_k  + \alpha ^{s - 1} G_{k - 1} } \right).
\end{equation}
Combining~\eqref{eq.qf73jtg} and~\eqref{eq.s6sm0bz} according to the Binet formula yields the following generalization of~\eqref{eq.zrwo56t}.
\begin{proposition}
If $s$ and $k$ are any integers, then
\begin{equation*}
G_{k + 1} H_{s + 1}  + G_{k - 2} H_{s - 2}  = 2G_k H_s  + 2G_{k - 1} H_{s - 1}.
\end{equation*}
\end{proposition}
Long's identities~\cite[(31)--(35)]{long86} are all special cases of the above proposition.

\bigskip

The reader is invited to apply the method (second component) to verify that the identity: 
\begin{equation*}
G_{k + 1}^3  = 3G_k^3  + 6G_{k - 1}^3  - 3G_{k - 2}^3  - G_{k - 3}^3,\quad \text{\cite[Equation (3)]{brousseau68}},
\end{equation*}
has the following generalization.
\begin{proposition}
If $k$, $r$ and $s$ are any integers, then
\begin{equation*}
\begin{split}
G_{k + 1} H_{r + 1} I_{s + 1}  &= 3G_k H_r I_s  + 6G_{k - 1} H_{r - 1} I_{s - 1}\\
&\qquad  - 3G_{k - 2} H_{r - 2} I_{s - 2}  - G_{k - 3} H_{r - 3} I_{s - 3},
\end{split}
\end{equation*}
where $(G_j)$, $(H_j)$ and $(I_j)$ are gibonacci sequences.
\end{proposition}
\subsubsection{Generalizations of a triple-angle identity of Lucas}
Differentiating the Fibonacci function form:
\begin{equation*}
f(3k)=f(k + 1)^3 + f(k)^3 - f(k - 1)^3,
\end{equation*}
of the identity (see Vorob'ev~\cite[p.~16]{vorobev61}):
\begin{equation}\label{eq.seu1n3o}
F_{3k}=F_{k + 1}^3 + F_k^3 - F_{k - 1}^3,
\end{equation}
gives, after using~\eqref{eq.mbcae9n} and~\eqref{eq.bf3c8fj},
\begin{equation*}
\beta ^{3k}  = F_{k + 1}^2 \beta ^{k + 1}  + F_k^2 \beta ^k  - F_{k - 1}^2 \beta ^{k - 1},
\end{equation*}
that is
\begin{equation}\label{eq.vlr5b97}
\beta ^{2k + s}  = F_{k + 1}^2 \beta ^{s + 1}  + F_k^2 \beta ^s  - F_{k - 1}^2 \beta ^{s - 1},
\end{equation}
where $s$ is an arbitrary integer; and also
\begin{equation}\label{eq.igzbpvl}
\alpha ^{2k + s}  = F_{k + 1}^2 \alpha ^{s + 1}  + F_k^2 \alpha ^s  - F_{k - 1}^2 \alpha ^{s - 1}.
\end{equation}
Combining~\eqref{eq.vlr5b97} and~\eqref{eq.igzbpvl}, we have the first generalization of~\eqref{eq.seu1n3o}.
\begin{proposition}
If $k$ and $s$ are any integers, then
\begin{equation}
G_{2k + s}  = F_{k + 1}^2 G_{s + 1}  + F_k^2 G_s  - F_{k - 1}^2 G_{s - 1}.
\end{equation}
\end{proposition}
The reader is invited to check that application of the method (second component) two more times with respect to $k$ gives the full generalization of~\eqref{eq.seu1n3o}, stated in Proposition~\ref{prop.nbs1udb}.
\begin{proposition}\label{prop.nbs1udb}
If $k$, $r$ and $s$ are any integers, then
\begin{equation*}
\begin{split}
&G_0 H_0 I_{k + r + s - 2}  + \left( {G_0 H_1  + G_1 H_0 } \right)I_{k + r + s - 1}  + G_1 H_1 I_{k + r + s} \\
&\qquad = G_{s + 1} H_{r + 1} I_{k + 1}  + G_s H_r I_k  - G_{s - 1} H_{r - 1} I_{k - 1},
\end{split}
\end{equation*}
where $(G_j)$, $(H_j)$ and $(I_j)$ are gibonacci sequences with initial terms $G_0$, $G_1$; $H_0$, $H_1$ and $I_0$, $I_1$.
\end{proposition}
In particular,
\begin{equation*}
G_0^2 G_{3k - 2}  + 2G_0 G_1 G_{3k - 1}  + G_1^2 G_{3k}  = G_{k + 1}^3  + G_k^3  - G_{k - 1}^3,
\end{equation*}
with the special cases:
\begin{gather*}
F_{3k}=F_{k + 1}^3 + F_k^3 - F_{k - 1}^3,\\
5L_{3k}=L_{k + 1}^3 + L_k^3 - L_{k - 1}^3.
\end{gather*}
Identities (42)--(45) of Long~\cite{long86} are special cases of the identity stated in Proposition~\ref{prop.nbs1udb}.
\subsubsection{Sum of products of the terms of two gibonacci sequences}
Simple telescoping of the identity
\begin{equation*}
G_{j + k}^2  = G_{j + k + 1}G_{j + k}   - G_{j + k} G_{j + k - 1},
\end{equation*}
gives
\begin{equation*}
\sum_{j = 1}^n {G_{j + k}^2 }  = G_{n + k} G_{n + k + 1}  - G_k G_{k + 1},
\end{equation*}
whose gibonacci function form is
\begin{equation*}
\sum_{j = 1}^n {g(j + k)^2 }  = g(n + k)g(n + k + 1) - g(k)g(k + 1).
\end{equation*}
Differentiating:
\begin{equation*}
\begin{split}
2\sum_{j = 1}^n {g(j + k)\frac{\partial }{{\partial k}}g(j + k)}  &= g(n + k)\frac{\partial }{{\partial k}}g(n + k + 1) + g(n + k + 1)\frac{\partial }{{\partial k}}g(n + k)\\
&\qquad - g(k)\frac{\partial }{{\partial k}}g(k + 1) - g(k + 1)\frac{\partial }{{\partial k}}g(k),
\end{split}
\end{equation*}
and, upon use of~\eqref{eq.brp6mp3} and~\eqref{eq.k32lr6u}, we have
\begin{equation*}
2\sum_{j = 1}^n {\beta ^{j + k} G_{j + k} }  = \beta ^{n + k + 1} G_{n + k}  + \beta ^{n + k} G_{n + k + 1}  - \beta ^k G_{k + 1}  - \beta ^{k + 1} G_k;
\end{equation*}
which, multiplying through by $\beta^{s - k}$, $s$ an arbitrary integer, gives
\begin{equation}\label{eq.b0ccikn}
2\sum_{j = 1}^n {\beta ^{j + s} G_{j + k} }  = \beta ^{n + s + 1} G_{n + k}  + \beta ^{n + s} G_{n + k + 1}  - \beta ^s G_{k + 1}  - \beta ^{s + 1} G_k,
\end{equation}
and hence also
\begin{equation}\label{eq.uax0hdc}
2\sum_{j = 1}^n {\alpha ^{j + s} G_{j + k} }  = \alpha ^{n + s + 1} G_{n + k}  + \alpha ^{n + s} G_{n + k + 1}  - \alpha ^s G_{k + 1}  - \alpha ^{s + 1} G_k.
\end{equation}
Combining~\eqref{eq.b0ccikn} and~\eqref{eq.uax0hdc} according to the Binet formulas yields the result stated in the next proposition.
\begin{proposition}
If $k$, $s$ and $n$ are any integers, then
\begin{equation}\label{eq.lyde8vc}
\sum_{j = 1}^n {G_{j + k} H_{j + s} }  = G_{n + k} H_{n + s + 1}  + G_{n + k + 1} H_{n + s}  - G_{k + 1} H_s  - G_k H_{s + 1}.
\end{equation}
\end{proposition}
Using generating function techniques, Berzsenyi~\cite{berzsenyi} found alternative expressions for the special case
\begin{equation*}
\sum_{j = 0}^{2n + \delta } {G_j G_{j + 2k + \delta } } ,\quad\text{$\delta=1$ or $0$}.
\end{equation*}
Long's identities~\cite[(4)--(7)]{long86} are alternative expressions for special cases of the above proposition.  Kronenburg~\cite[Identity (11.1)]{kronenburg21} also derived~\eqref{eq.lyde8vc}.
\subsubsection{Generalizations of Hoggatt and Bicknell's identity~\eqref{eq.iigew21}}
Taking the imaginary part of~\eqref{eq.vzdatpd} on page~\pageref{eq.vzdatpd}, using~\eqref{eq.mbcae9n}, we have
\begin{equation*}
\begin{split}
&\sum_{j = 0}^{4n + 1} {( - 1)^{j - 1} \binom{4n + 1}jF_{j + k}^3 \Im\frac{\partial }{{\partial k}}f(j + k)}\\ 
&\qquad = 25^n \left( {F_{2n + k + 1}^3 \Im\frac{\partial }{{\partial k}}f(2n + k + 1) - F_{2n + k}^3 \Im\frac{\partial }{{\partial k}}f(2n + k)} \right);
\end{split}
\end{equation*}
which, applying~\eqref{eq.bf3c8fj}, gives
\begin{equation*}
\sum_{j = 0}^{4n + 1} {( - 1)^{j - 1} \binom{4n + 1}jF_{j + k}^3 \beta^{j + k}} = 25^n \left( {F_{2n + k + 1}^3 \beta^{2n + k + 1} - F_{2n + k}^3 \beta^{2n + k}} \right),
\end{equation*}
and hence
\begin{equation}\label{eq.lxbgoc2}
\sum_{j = 0}^{4n + 1} {( - 1)^{j - 1} \binom{4n + 1}jF_{j + k}^3 \beta^{j + r}} = 25^n \left( {F_{2n + k + 1}^3 \beta^{2n + r + 1} - F_{2n + k}^3 \beta^{2n + r}} \right);
\end{equation}
where $r$ is an arbitrary integer.

\bigskip

Interchanging $\alpha$ and $\beta$ in~\eqref{eq.lxbgoc2}, we also have
\begin{equation}\label{eq.ovw7s6k}
\sum_{j = 0}^{4n + 1} {( - 1)^{j - 1} \binom{4n + 1}jF_{j + k}^3 \alpha^{j + r}} = 25^n \left( {F_{2n + k + 1}^3 \alpha^{2n + r + 1} - F_{2n + k}^3 \alpha^{2n + r}} \right).
\end{equation}
Combining~\eqref{eq.lxbgoc2} and~\eqref{eq.ovw7s6k} according to the Binet formula, we have the result stated in the next proposition.
\begin{proposition}
If $n$ is a non-negative integer and $k$ is any integer, then
\begin{equation}\label{eq.p7si1ht}
\sum_{j = 0}^{4n + 1} {( - 1)^{j - 1} \binom{4n + 1}jF_{j + k}^3 G_{j + r}} = 25^n \left( {F_{2n + k + 1}^3 G_{2n + r + 1} - F_{2n + k}^3 G_{2n + r}} \right).
\end{equation}
\end{proposition}
Observe that~\eqref{eq.iigew21} and~\eqref{eq.f6y5qaq} are particular cases of~\eqref{eq.p7si1ht}.

\bigskip

By applying the method (second component) three more times, with $k$ as the index of interest, the reader is invited to establish the further generalization presented in the next proposition.
\begin{proposition}\label{prop.ri7x357}
If $n$ is a non-negative integer and $r$, $k$ and $s$ are any integers, then
\begin{equation*}
\begin{split}
&\sum_{j = 0}^{4n + 1} {( - 1)^{j - 1} \binom{4n + 1}jE_{j + k} G_{j + r} H_{j + s}I_{j + t} }\\
&\qquad = 25^n \left( {E_{2n + k + 1} G_{2n + r + 1} H_{2n + s + 1}I_{2n + t + 1}  - E_{2n + k} G_{2n + r} H_{2n + s}I_{2n + t} } \right);
\end{split}
\end{equation*}
where $(E)_{j\in\mathbb Z}$, $(G)_{j\in\mathbb Z}$, $(H)_{j\in\mathbb Z}$ and $(I)_{j\in\mathbb Z}$ are gibonacci sequences with seeds $E_0$, $E_1$; $G_0$, $G_1$; $H_0$, $H_1$ and $I_0$, $I_1$.
\end{proposition}
\subsubsection{Generalizations of an identity of Melham}
In this section, we present a generalization of Melham's identity,~\eqref{eq.zi7ygj7} on page~\pageref{eq.zi7ygj7}, that is
\begin{equation*}
6\left( {\sum_{j = 0}^{2n - 1} {G_{k + j}^2 } } \right)^2  = F_{2n}^2 \left( {G_{k + n - 2}^4  + 4G_{k + n - 1}^4  + 4G_{k + n}^4  + G_{k + n + 1}^4 } \right);
\end{equation*}
whose Fibonacc-gibonacci function derivative is
\begin{equation*}
\begin{split}
&\left( {6\sum_{j = 0}^{2n - 1} {g(k + j)^2 } } \right)\sum_{j = 0}^{2n - 1} {g(k + j)\frac{\partial }{{\partial k}}g(k + j)} \\
&\qquad= f(2n)^2 \left( {g(k + n - 2)^3 \frac{\partial }{{\partial k}}g(k + n - 2)} \right. + 4g(k + n - 1)^3 \frac{\partial }{{\partial k}}g(k + n - 1)\\
&\qquad\left. { + 4g(k + n)^3 \frac{\partial }{{\partial k}}g(k + n) + g(k + n + 1)^3 \frac{\partial }{{\partial k}}g(k + n + 1)} \right).
\end{split}
\end{equation*}
Taking the imaginary part and applying~\eqref{eq.brp6mp3} and~\eqref{eq.k32lr6u} leads to
\begin{equation*}
\begin{split}
&6\sum_{j = 0}^{2n - 1} {G_{k + j}^2 } \sum_{j = 0}^{2n - 1} {\beta ^{k + j} G_{k + j} }\\ 
&\qquad = F_{2n}^2 \left( {\beta ^{k + n - 2} G_{k + n - 2}^3  + 4\beta ^{k + n - 1} G_{k + n - 1}^3  + 4\beta ^{k + n} G_{k + n}^3  + \beta ^{k + n + 1} G_{k + n + 1}^3 } \right),
\end{split}
\end{equation*}
which multiplying through by $\beta^{r - k}$, $r$ an arbitrary integer, gives
\begin{equation}\label{eq.xaxq153}
\begin{split}
&6\sum_{j = 0}^{2n - 1} {G_{k + j}^2 } \sum_{j = 0}^{2n - 1} {\beta ^{r + j} G_{k + j} } \\
&\qquad = F_{2n}^2 \left( {\beta ^{r + n - 2} G_{k + n - 2}^3  + 4\beta ^{r + n - 1} G_{k + n - 1}^3  + 4\beta ^{r + n} G_{k + n}^3  + \beta ^{r + n + 1} G_{k + n + 1}^3 } \right),
\end{split}
\end{equation}
and also
\begin{equation}\label{eq.uqsnbmg}
\begin{split}
&6\sum_{j = 0}^{2n - 1} {G_{k + j}^2 } \sum_{j = 0}^{2n - 1} {\alpha ^{r + j} G_{k + j} } \\
&\qquad = F_{2n}^2 \left( {\alpha ^{r + n - 2} G_{k + n - 2}^3  + 4\alpha ^{r + n - 1} G_{k + n - 1}^3  + 4\alpha ^{r + n} G_{k + n}^3  + \alpha ^{r + n + 1} G_{k + n + 1}^3 } \right).
\end{split}
\end{equation}
Combining~\eqref{eq.xaxq153} and~\eqref{eq.uqsnbmg} according to the Binet formula yields the next result.
\begin{proposition}\label{prop.fms8uy6}
If $k$ and $r$ are any integers, then
\begin{equation}
\begin{split}
&6\sum_{j = 0}^{2n - 1} {G_{k + j}^2 } \sum_{j = 0}^{2n - 1} {H_{r + j} G_{k + j} } \\
&\qquad = F_{2n}^2 \left( {H_{r + n - 2} G_{k + n - 2}^3  + 4H_{r + n - 1} G_{k + n - 1}^3  + 4H_{r + n} G_{k + n}^3  + H_{r + n + 1} G_{k + n + 1}^3 } \right).
\end{split}
\end{equation}

\end{proposition}
Repeated application of the method (second component) two more times  to the identity stated in Proposition~\ref{prop.fms8uy6} with $k$ as the index of interest establishes the next result.
\begin{proposition}
If $k$, $r$, $s$ and $t$ are any integers, then
\begin{equation*}
\begin{split}
&2\sum_{j = 0}^{2n - 1} {G_{k + j} J_{t + j} } \sum_{j = 0}^{2n - 1} {H_{r + j} I_{s + j} }  + 2\sum_{j = 0}^{2n - 1} {I_{s + j} J_{t + j} } \sum_{j = 0}^{2n - 1} {G_{k + j} H_{r + j} } \\
&\qquad\qquad + 2\sum_{j = 0}^{2n - 1} {G_{k + j} I_{s + j} } \sum_{j = 0}^{2n - 1} {H_{r + j} J_{t + j} } \\
& = F_{2n}^2 \left( {G_{k + n - 2} H_{r + n - 2} I_{s + n - 2} J_{t + n - 2}  + 4\,G_{k + n - 1} H_{r + n - 1} I_{s + n - 1} J_{t + n - 1} } \right.\\
&\qquad\qquad\left. { + 4\,G_{k + n} H_{r + n} I_{s + n} J_{t + n}  + G_{k + n + 1} H_{r + n + 1} I_{s + n + 1} J_{t + n + 1} } \right);
\end{split}
\end{equation*}
where $(G_j)$, $(H_j)$, $(I_j)$ and $(J_j)$ are gibonacci sequences.
\end{proposition}
\subsubsection{Generalizations of an identity of Howard}
Taking the imaginary part of~\eqref{eq.vig7d8s} on page~\pageref{eq.vig7d8s} gives
\begin{equation}\label{eq.f6g296p}
G_{k + r}\Im\frac{d}{{ds}}f(s) + ( - 1)^{r - 1} G_k\Im\frac{\partial }{{\partial s}}f(s - r) = F_r\Im\frac{\partial }{{\partial s}}g(k + s);
\end{equation}
which, on using~\eqref{eq.k32lr6u}, yields
\begin{equation}\label{eq.ybzjodo}
\beta ^s G_{k + r}  + ( - 1)^{r - 1} \beta ^{s - r} G_k  = \beta ^{k + s - 1} G_0 F_r  + \beta ^{k + s} G_1 F_r
\end{equation}
and hence, also
\begin{equation}\label{eq.v554f0r}
\alpha ^s G_{k + r}  + ( - 1)^{r - 1} \alpha ^{s - r} G_k  = \alpha ^{k + s - 1} G_0 F_r  + \alpha ^{k + s} G_1 F_r.
\end{equation}
Combining~\eqref{eq.ybzjodo} and~\eqref{eq.v554f0r} provides the following generalization of Howard's identity~\eqref{eq.e715lgu}.
\begin{proposition}
If $r$, $s$ and $k$ are any integers, then
\begin{equation*}
H_s G_{k + r}  + ( - 1)^{r - 1} H_{s - r} G_k  = F_r \left( {G_0 H_{k + s - 1} + G_1 H_{k + s} } \right).
\end{equation*}
\end{proposition}
By taking the imaginary part of~\eqref{eq.dzcinsk} on page~\pageref{eq.dzcinsk}, the reader is invited to establish the result stated in Proposition~\ref{prop.o7pee2v}.
\begin{proposition}\label{prop.o7pee2v}
If $k$, $r$, $s$ and $t$ are integers, then
\begin{equation*}
\begin{split}
&F_s \left( {G_0 H_{k + r + t - 1}  + G_1 H_{k + r + t} } \right) - ( - 1)^r F_{s - r} G_k \left( {H_{t + 1}  + H_{t - 1} } \right)\\
&\qquad + ( - 1)^r G_k H_{s - r + t} \\
&\qquad = G_{k + s} H_{r + t}. 
\end{split}
\end{equation*}
\end{proposition}
\subsubsection{Generalizations of Candido's identity}
Taking the imaginary part of~\eqref{eq.eqypotv} on page~\pageref{eq.eqypotv} according to the prescription of~\eqref{eq.brp6mp3} and~\eqref{eq.k32lr6u}, we have
\begin{equation*}
\begin{split}
&2\left( {\beta ^k G_k^3  + \beta ^{k + 1} G_{k + 1}^3  + \beta ^{k + 2} G_{k + 2}^3 } \right)\\
&\qquad = \left( {G_k^2  + G_{k + 1}^2  + G_{k + 2}^2 } \right)\left( {\beta ^k G_k  + \beta ^{k + 1} G_{k + 1}  + \beta ^{k + 2} G_{k + 2} } \right);
\end{split}
\end{equation*}
which, upon multiplication by $\beta^{r - k}$ gives
\begin{equation}\label{eq.kwuc7to}
\begin{split}
&2\left( {\beta ^r G_k^3  + \beta ^{r + 1} G_{k + 1}^3  + \beta ^{r + 2} G_{k + 2}^3 } \right)\\
&\qquad = \left( {G_k^2  + G_{k + 1}^2  + G_{k + 2}^2 } \right)\left( {\beta ^r G_k  + \beta ^{r + 1} G_{k + 1}  + \beta ^{r + 2} G_{k + 2} } \right);
\end{split}
\end{equation}
which, on account of Proposition~\ref{prop.invariant}, also implies
\begin{equation}\label{eq.w6hqycw}
\begin{split}
&2\left( {\alpha ^r G_k^3  + \alpha ^{r + 1} G_{k + 1}^3  + \alpha ^{r + 2} G_{k + 2}^3 } \right)\\
&\qquad = \left( {G_k^2  + G_{k + 1}^2  + G_{k + 2}^2 } \right)\left( {\alpha ^r G_k  + \alpha ^{r + 1} G_{k + 1}  + \alpha ^{r + 2} G_{k + 2} } \right).
\end{split}
\end{equation}
Combining according to the Binet formula gives the following generalization of Candido's identity.
\begin{proposition}
If $r$ and $k$ are any integers, then
\begin{equation}\label{eq.y80lcc4}
\begin{split}
&2\left( {H_rG_k^3  + H_{r + 1} G_{k + 1}^3  + H_{r + 2} G_{k + 2}^3 } \right)\\
&\qquad = \left( {G_k^2  + G_{k + 1}^2  + G_{k + 2}^2 } \right)\left( {H_r G_k  + H_{r + 1} G_{k + 1}  + H_{r + 2} G_{k + 2} } \right).
\end{split}
\end{equation}
\end{proposition}
By differentiating the gibonacci function form of~\eqref{eq.y80lcc4} three more times with respect to $k$, the reader is invited to demonstrate the further generalization of the Candido identity stated in Proposition~\ref{prop.wzey6tg}.
\begin{proposition}\label{prop.wzey6tg}
If $k$, $r$, $s$ and $t$ are any integers, then
\begin{equation*}
\begin{split}
&6\left( {G_k H_r M_s N_t  + G_{k + 1} H_{r + 1} M_{s + 1} N_{t + 1}  + G_{k + 2} H_{r + 2} M_{s + 2} N_{t + 2} } \right)\\
& = \left( {G_k M_s  + G_{k + 1} M_{s + 1}  + G_{k + 2} M_{s + 2} } \right)\left( {H_r N_t  + H_{r + 1} N_{t + 1}  + H_{r + 2} N_{t + 2} } \right)\\
&\qquad + \left( {G_k H_r  + G_{k + 1} H_{r + 1}  + G_{k + 2} H_{r + 2} } \right)\left( {M_s N_t  + M_{s + 1} N_{t + 1}  + M_{s + 2} N_{t + 2} } \right)\\
&\qquad + \left( {G_k N_t  + G_{k + 1} N_{t + 1}  + G_{k + 2} N_{t + 2} } \right)\left( {H_r M_s  + H_{r + 1} M_{s + 1}  + H_{r + 2} M_{s + 2} } \right),
\end{split}
\end{equation*}
where ${(M_j)}_{j\in\mathbb Z}$ and ${(N_j)}_{j\in\mathbb Z}$ are gibonacci sequences with seeds $M_0$ and $M_1$ and $N_0$ and $N_1$.
\end{proposition}
In particular, we have
\begin{equation*}
\begin{split}
&6\left( {F_k F_r F_s F_t  + F_{k + 1} F_{r + 1} F_{s + 1} F_{t + 1}  + F_{k + 2} F_{r + 2} F_{s + 2} F_{t + 2} } \right)\\
& = \left( {F_k F_s  + F_{k + 1} F_{s + 1}  + F_{k + 2} F_{s + 2} } \right)\left( {F_r F_t  + F_{r + 1} F_{t + 1}  + F_{r + 2} F_{t + 2} } \right)\\
&\qquad + \left( {F_k F_r  + F_{k + 1} F_{r + 1}  + F_{k + 2} F_{r + 2} } \right)\left( {F_s F_t  + F_{s + 1} F_{t + 1}  + F_{s + 2} F_{t + 2} } \right)\\
&\qquad + \left( {F_k F_t  + F_{k + 1} F_{t + 1}  + F_{k + 2} F_{t + 2} } \right)\left( {F_r F_s  + F_{r + 1} F_{s + 1}  + F_{r + 2} F_{s + 2} } \right).
\end{split}
\end{equation*}
\subsection{Generalizations of a Lucas number identity}
Differentiating the Lucas function:
\begin{equation*}
l(2r) + 2(-1)^r=l(r)^2,
\end{equation*}
of the well-known identity:
\begin{equation*}
L_{2r} + 2(-1)^r = L_r^2,
\end{equation*}
gives
\begin{equation*}
\frac{d}{{dr}}l(2r) + ( - 1)^r i\pi  = l(r)\frac{d}{{dr}}l(r),
\end{equation*}
which, employing~\eqref{eq.p4bbs37} and~\eqref{eq.b00ifyn}, yields
\begin{equation*}
\beta^{2r} + (-1)^r =\beta^rL_r,
\end{equation*}
or multiplying through by $\beta^{s - r}$,
\begin{equation}\label{eq.x5gev45}
\beta^{s + r} + (-1)^r\beta^ {s - r} =\beta^sL_r,
\end{equation}
and also
\begin{equation}\label{eq.vlq8mj5}
\alpha^{s + r} + (-1)^r\alpha^ {s - r} =\alpha^sL_r,
\end{equation}
Combining~\eqref{eq.x5gev45} and~\eqref{eq.vlq8mj5} according to the Binet formula gives the following multiplication formula (also Vajda~\cite[Formula (10a)]{vajda}).
\begin{proposition}
If $r$ and $s$ are any integers, then
\begin{equation*}
G_{s + r} + (-1)^rG_{s - r}=L_rG_s.
\end{equation*}
\end{proposition}
Differentiating the function
\begin{equation*}
g(s + r) + ( - 1)^r g(s - r) = l(r)g(s),
\end{equation*}
with respect to $r$ gives
\begin{equation*}
\begin{split}
&\frac{\partial }{{\partial r}}g(s + r) + ( - 1)^r \pi ig(s - r) - ( - 1)^r \frac{\partial }{{\partial r}}g(s - r)\\
&\qquad= g(s)\frac{d}{{dr}}l(r),
\end{split}
\end{equation*}
and taking the imaginary part, making use of~\eqref{eq.b00ifyn},~\eqref{eq.brp6mp3} and~\eqref{eq.k32lr6u}, we find
\begin{equation*}
\begin{split}
&G_0 \left( {\beta ^{k + r + s - 1}  - ( - 1)^r \beta ^{k + s - r - 1} } \right) + G_1 \left( {\beta ^{k + s - 1 + r}  - ( - 1)^r \beta ^{k + s - 1 - r} } \right)\\
&\qquad = ( - 1)^r \beta ^k G_{s - r} \sqrt 5  - \beta ^{r + k} G_s \sqrt 5 ,
\end{split}
\end{equation*}
and also
\begin{equation*}
\begin{split}
&G_0 \left( {\alpha ^{k + r + s - 1}  - ( - 1)^r \alpha ^{k + s - r - 1} } \right) + G_1 \left( {\alpha ^{k + s - 1 + r}  - ( - 1)^r \alpha ^{k + s - 1 - r} } \right)\\
&\qquad = ( - 1)^{r - 1} \alpha ^k G_{s - r} \sqrt 5  + \alpha ^{r + k} G_s \sqrt 5, 
\end{split}
\end{equation*}
and hence the next result.
\begin{proposition}
If $r$, $k$ and $s$ are integers, then
\begin{equation*}
\begin{split}
&G_s \left( {H_{r + k + 1}  + H_{r + k - 1} } \right) + ( - 1)^{r - 1} G_{s - r} \left( {H_{k + 1}  + H_{k - 1} } \right)\\
&\qquad =G_0 F_r \left( {H_{k + s}  + H_{k + s - 2} } \right) + G_1 F_r \left( {H_{k + s + 1}  + H_{k + s - 1} } \right).
\end{split}
\end{equation*}
\end{proposition}
Note that we used
\begin{equation*}
G_{s + r}  - ( - 1)^r G_{s - r}  = F_r \left( {G_{s + 1}  + G_{s - 1} } \right),\quad\text{\cite[Formula (10b)]{vajda}}.
\end{equation*}
Note also that $(H_j)=(L_j)$ in the proposition gives the Howard identity~\eqref{eq.e715lgu} while \mbox{$(H_j)=(F_j)$} produces~\eqref{eq.nh7xbr5}.
\begin{remark}
It should be noted that in order that the results obtained from taking the imaginary part be valid, it is necessary that the (generalized) Fibonacci function identity obtained from the original (generalized) Fibonacci number identity holds for all real $x$; not just integers. This point is taken into consideration in the examples presented in \S~\ref{sec.b1ek8kn} to~\ref{sec.n3q9dr5}.
\end{remark}

\subsubsection{A generalization of the Gelin-Ces\`aro identity}\label{sec.b1ek8kn}
Since
\begin{equation}
G_{k - 1} G_{k + 1}  = G_k^2  - ( - 1)^k e_G,\quad\text{\cite[Identity 28]{vajda}},
\end{equation}
where $e_G=G_0^2 - G_1^2 +G_1G_0$, and
\begin{equation*}
\begin{split}
G_{k - 2} G_{k + 2}  &= \left( {G_k  - G_{k - 1} } \right)\left( {G_k  + G_{k + 1} } \right)\\
& = G_k^2  + (G_k G_{k + 1}  - G_{k - 1} G_k ) - G_{k - 1} G_{k + 1} \\
& = G_k^2  + G_k^2  - G_{k - 1} G_{k + 1} \\
& = G_k^2  + G_k^2  - \left( {G_k^2  - ( - 1)^k e_G } \right)\\
& = G_k^2  + ( - 1)^k e_G ;
\end{split}
\end{equation*}
we have the following generalization of the Gelin-Ces\`aro identity:
\begin{equation*}
G_{k - 2} G_{k - 1} G_{k + 1} G_{k + 2}  = G_k^4  - ( - 1)^{2k} e_G^2, 
\end{equation*}
where we have retained $( - 1)^{2k}$ to allow a direct conversion to the gibonacci function form which is required to hold for all real numbers $k$, namely,
\begin{equation}\label{eq.r05dxhq}
g(k - 2) g(k - 1) g(k + 1)g(k + 2)  = g(k)^4  - ( - 1)^{2k} e_G^2 .
\end{equation}
Differentiating~\eqref{eq.r05dxhq} gives
\begin{equation*}
\begin{split}
&\frac{d}{{dk}}g(k - 2)g(k - 1)g(k + 1)g(k + 2)\\
&\quad + g(k - 2)\frac{d}{{dk}}g(k - 1)g(k + 1)g(k + 2)\\
&\quad\, + g(k - 2)g(k - 1)\frac{d}{{dk}}g(k + 1)g(k + 2)\\
&\quad\; + g(k - 2)g(k - 1)g(k + 1)\frac{d}{{dk}}g(k + 2)\\
&\quad\;\, = 4g(k)^3 \frac{d}{{dk}}g(k) - 2i( - 1)^{2k} \pi e_G^2 ;
\end{split}
\end{equation*}
so that taking the imaginary part, we have
\begin{equation*}
\begin{split}
&B_G \beta ^{k - 2} G_{k - 1} G_{k + 1} G_{k + 2}  + G_{k - 2} B_G \beta ^{k - 1} G_{k + 1} G_{k + 2}\\ 
&\qquad + G_{k - 2} G_{k - 1} B_G \beta ^{k + 1} G_{k + 2}  + G_{k - 2} G_{k - 1} G_{k + 1} B_G \beta ^{k + 2} \\
& = 4G_k^3 B_G \beta ^k  - 2( - 1)^{2k} e_G^2; 
\end{split}
\end{equation*}
and substituting $(G_0\alpha - G_1)/\sqrt 5$ for $B_G$ from~\eqref{eq.k32lr6u} and multiplying through by $\beta^r$ yields
\begin{equation}\label{eq.c4fuwto}
\begin{split}
&\beta ^{k + r - 2} G_{k - 1} G_{k + 1} G_{k + 2}  + G_{k - 2} \beta ^{k + r - 1} G_{k + 1} G_{k + 2} \\
&\quad + G_{k - 2} G_{k - 1} \beta ^{k + r + 1} G_{k + 2}  + G_{k - 2} G_{k - 1} G_{k + 1} \beta ^{k + r + 2} \\
& = 4G_k^3 \beta ^{k + r}  + 2\left( {\beta ^{r + 1} G_0  - \beta ^r G_1 } \right)e_G \sqrt 5,
\end{split}
\end{equation}
which also implies
\begin{equation}\label{eq.r12zjzv}
\begin{split}
&\alpha ^{k + r - 2} G_{k - 1} G_{k + 1} G_{k + 2}  + G_{k - 2} \alpha ^{k + r - 1} G_{k + 1} G_{k + 2} \\
&\quad + G_{k - 2} G_{k - 1} \alpha ^{k + r + 1} G_{k + 2}  + G_{k - 2} G_{k - 1} G_{k + 1} \alpha ^{k + r + 2} \\
& = 4G_k^3 \alpha ^{k + r}  - 2\left( {\alpha ^{r + 1} G_0  - \alpha ^r G_1 } \right)e_G \sqrt 5 .
\end{split}
\end{equation}
Note that we used
\begin{equation*}
\frac{1}{{B_G }} = \frac{{\sqrt 5 }}{{G_0 \alpha  - G_1 }} = \frac{{\sqrt 5 \left( {G_0 \beta  - G_1 } \right)}}{{ - e_G }}.
\end{equation*}
Combining~\eqref{eq.c4fuwto} and~\eqref{eq.r12zjzv}, we arrive at the generalization of the Gelin-Ces\`aro identity stated in Proposition~\ref{prop.om4cjop}.
\begin{proposition}\label{prop.om4cjop}
If $k$ and $r$ are any integers, then
\begin{equation*}
\begin{split}
&H_{k + r - 2} G_{k - 1} G_{k + 1} G_{k + 2}  + G_{k - 2} H_{k + r - 1} G_{k + 1} G_{k + 2} \\
&\quad + G_{k - 2} G_{k - 1} H_{k + r + 1} G_{k + 2}  + G_{k - 2} G_{k - 1} G_{k + 1} H_{k + r + 2} \\
& = 4H_{k + r} G_k^3  - 2e_G G_0 \left( {H_{r + 2}  + H_r } \right) + 2e_G G_1 \left( {H_{r + 1}  + H_{r - 1} } \right).
\end{split}
\end{equation*}
\end{proposition}
\subsubsection{Generalizations of Catalan's identity}
Upon differentiating the Fibonacci function form:
\begin{equation*}
f(k - r)f(k + r) = f(k)^2  + ( - 1)^{k + r + 1} f(r)^2,
\end{equation*}
of Catalan's identity:
\begin{equation*}
F_{k - r} F_{k + r}  = F_k^2  + ( - 1)^{k + r + 1} F_r^2,
\end{equation*}
with respect to $k$ and applying the prescription~\eqref{eq.mbcae9n} and~\eqref{eq.bf3c8fj}, we obtain
\begin{equation}\label{eq.o8e2j38}
F_{k + r} \beta ^{k - r + s}  + F_{k - r} \beta ^{k + r + s}  = 2F_k \beta ^{k + s}  + ( - 1)^{k + r} \beta ^s F_r^2 \sqrt 5,
\end{equation} 
and hence also
\begin{equation}\label{eq.d6vmxww}
F_{k + r} \alpha ^{k - r + s}  + F_{k - r} \alpha ^{k + r + s}  = 2F_k \alpha ^{k + s}  - ( - 1)^{k + r} \alpha ^s F_r^2 \sqrt 5.
\end{equation}
Combining~\eqref{eq.o8e2j38} and~\eqref{eq.d6vmxww} according to the Binet formula gives the following generalization of Catalan's identity.
\begin{proposition}
If $k$, $r$ and $s$ are any integers, then
\begin{equation}\label{eq.ghaetrs}
F_{k + r} G_{k - r + s}  + F_{k - r} G_{k + r + s}  = 2F_k G_{k + s}  + ( - 1)^{k + r + 1} F_r^2 \left( {G_{s + 1}  + G_{s - 1} } \right).
\end{equation}
\end{proposition}
Writing
\begin{equation}\label{eq.esegafn}
F_{k - s + r} G_{k - r}  + F_{k - s - r} G_{k + r}  = 2F_{k - s} G_k  + ( - 1)^{k - s + r + 1} F_r^2 \left( {G_{s + 1}  + G_{s - 1} } \right),
\end{equation}
and setting $k=s$ gives the multiplication formula
\begin{equation}\label{eq.rel90s5}
G_{s + r}  - ( - 1)^r G_{s - r}  = F_r \left( {G_{s + 1}  + G_{s - 1} } \right),
\end{equation}
derived also by Vajda~\cite[Formula (10b)]{vajda}.

\bigskip

Applying the method (second component) to~\eqref{eq.rel90s5} with $r$ as the index of interest gives the next result.
\begin{proposition}
If $r$, $s$ and $t$ are any integers, then
\begin{equation}
\begin{split}
&L_r\left(G_0 H_{s + t - 1}  + G_1 H_{s + t}\right)  - ( - 1)^r G_{s - r} \left( {H_{t + 1}  + H_{t - 1} } \right)\\
&\qquad = H_{r + t} \left( {G_{s + 1}  + G_{s - 1} } \right).
\end{split}
\end{equation}
\end{proposition}
In particular, setting $t=s$, $(H_j)=(G_j)$ and the use of~(\cite[Formula (10a)]{vajda}):
\begin{equation*}
G_{n + m} + (-1)^mG_{n - m}=L_mG_n,
\end{equation*}
we have
\begin{equation*}
G_0G_{2s - 1}+G_1G_{2s}=G_s(G_{s + 1} + G_{s - 1}).
\end{equation*}
Differentiating the Fibonacci function form of~\eqref{eq.esegafn}, namely,
\begin{equation*}
\begin{split}
&f(k - s + r) g(k - r)  + f(k - s - r) g(k + r)\\
&\qquad  = 2f(k - s) g(k)  + ( - 1)^{k - s + r + 1} f(r)^2 \left( {g(s + 1)  + g(s - 1) } \right),
\end{split}
\end{equation*}
with respect to $s$, taking the imaginary part and making use of~\eqref{eq.mbcae9n},~\eqref{eq.bf3c8fj},~\eqref{eq.brp6mp3} and~\eqref{eq.k32lr6u} yield the identity stated in Proposition~\eqref{prop.w1edoil}.
\begin{proposition}\label{prop.w1edoil}
If $k$, $r$, $s$ and $t$ are any integers, then
\begin{equation}
\begin{split}
&G_{k + s - r} H_{k + r + t}  + G_{k + s + r} H_{k - r + t}\\
&\qquad= 2G_{k + s} H_{k + t} \\
&\qquad\; + ( - 1)^{k + r} F_r^2 \left( {G_0 \left( {H_{s + t}  + H_{s + t - 2} } \right) + G_1 \left( {H_{s + t + 1}  + H_{s + t - 1} } \right)} \right)\\
&\qquad\;\;- ( - 1)^{k + r} F_r^2 \left( {G_{s + 1}  + G_{s - 1} } \right)\left( {H_{t + 1}  + H_{t - 1} } \right).
\end{split}
\end{equation}
\end{proposition}
In particular,
\begin{equation}
G_{k - r}G_{k + r} - G_k^2=(-1)^{k + r}F_r^2e_G.
\end{equation}
\subsubsection{Generalization of an identity obtained from an inverse tangent relation}\label{sec.n3q9dr5}
The method (second component) cannot be applied to~\eqref{eq.k9bx5vw} because~\eqref{eq.hpngjml} on page~\pageref{eq.hpngjml} is valid only for integers $k$. In order to redeem the situation, we proceed as follows:
\begin{equation*}
\begin{split}
\tan ^{ - 1} \frac{1}{{F_{2k} }} - \tan ^{ - 1} \frac{1}{{F_{2k + 2} }} &= \tan ^{ - 1} \frac{{F_{2k + 2}  - F_{2k} }}{{F_{2k} F_{2k + 2}  + 1}}\\
&= \tan ^{ - 1} \frac{{F_{2k + 1} }}{{F_{2k + 1}^2  + ( - 1)^{2k + 1}  + 1}},
\end{split}
\end{equation*}
where our refrain from setting $( - 1)^{2k + 1}=-1$ ensures that the Fibonacci function form:
\begin{equation}\label{eq.mscmyfm}
\tan ^{ - 1} \frac{1}{{f(2k) }} - \tan ^{ - 1} \frac{1}{{f(2k + 2) }} = \tan ^{ - 1} \frac{f(2k + 1)}{{f(2k + 1)^2 +( - 1)^{2k + 1} + 1}},
\end{equation}
holds for all real numbers $k$.

\bigskip

Differentiating~\eqref{eq.mscmyfm} with respect to $k$ and taking the imaginary part, we find
\begin{equation*}
\begin{split}
&\frac{1}{{F_{2k}^2  + 1}}\Im\frac{d}{{dk}}f(2k) - \frac{1}{{F_{2k + 2}^2  + 1}}\Im\frac{d}{{dk}}f(2k + 2)\\
&\qquad = \frac{1}{{F_{2k + 1}^2  + 1}}\Im\frac{d}{{dk}}f(2k + 1) - \frac{\pi }{{F_{2k + 1} \left( {F_{2k + 1}^2  + 1} \right)}},
\end{split}
\end{equation*}
which, upon use of~\eqref{eq.mbcae9n} and~\eqref{eq.bf3c8fj}, gives
\begin{equation*}
\frac{{\beta ^{2k} }}{{F_{2k}^2  + 1}} - \frac{{\beta ^{2k + 2} }}{{F_{2k + 2}^2  + 1}} = \frac{{\beta ^{2k + 1} }}{{F_{2k + 1}^2  + 1}} + \frac{{\sqrt 5 }}{{F_{2k + 1} \left( {F_{2k + 1}^2  + 1} \right)}},
\end{equation*}
that is
\begin{equation}\label{eq.dol6t1x}
\frac{{\alpha ^{r + 2} }}{{F_{2k}^2  + 1}} - \frac{{\alpha ^r }}{{F_{2k + 2}^2  + 1}} = \frac{{\alpha ^{2k + r + 2} \sqrt 5 }}{{F_{2k + 1} \left( {F_{2k + 1}^2  + 1} \right)}} - \frac{{\alpha ^{r + 1} }}{{F_{2k + 1}^2  + 1}},
\end{equation}
where $r$ is an arbitrary integer and also
\begin{equation}\label{eq.lkj26sf}
\frac{{\beta ^{r + 2} }}{{F_{2k}^2  + 1}} - \frac{{\beta ^r }}{{F_{2k + 2}^2  + 1}} = -\frac{{\beta ^{2k + r + 2} \sqrt 5 }}{{F_{2k + 1} \left( {F_{2k + 1}^2  + 1} \right)}} - \frac{{\beta ^{r + 1} }}{{F_{2k + 1}^2  + 1}}.
\end{equation}
By combining~\eqref{eq.dol6t1x} and~\eqref{eq.lkj26sf}, we have the next result.
\begin{proposition}
If $r$ and $k$ are any integers, then
\begin{equation*}
\frac{{G_{r + 2} }}{{F_{2k}^2  + 1}} - \frac{{G_r }}{{F_{2k + 2}^2  + 1}}=\frac{{G_{2k + r + 3}  + G_{2k + r + 1} }}{{F_{2k + 1} (F_{2k + 1}^2  + 1)}} - \frac{{G_{r + 1} }}{{F_{2k + 1}^2  + 1}}.
\end{equation*}
\end{proposition}
\section{Extension of the method (second component) to the Horadam sequence}\label{sec.horadam}
The Horadam sequence $(W_j) = \left(W_j(W_0,W_1; p,q)\right)$ is defined, for all integers and arbitrary real numbers $W_0$, $W_1$, $p\ne 0$ and $q\ne 0$, by the recurrence relation
\begin{equation}\label{eq.c18xwge}
W_j  = pW_{j - 1}  - qW_{j - 2}, \quad  j \ge 2,
\end{equation}
with $W_{ - j}  =\left(pW_{-j + 1} - W_{-j + 2}\right)/q$.

\bigskip

\bigskip

Associated with  $(W_j)$ are the Lucas sequences of the first kind, $(U_j(p,q))=(W_j(0,1;p,q))$, and of the second kind, $(V_j(p,q))=(W_j(2,p;p,q))$; so that 
\begin{equation}\label{eq.dv6jotq}
U_0  = 0,\,U_1  = 1,\quad U_j  = pU_{j - 1}  - qU_{j - 2}, \quad  j \ge 2,
\end{equation}
and
\begin{equation}\label{eq.ynn0ffo}
V_0  = 2,\,V_1  = p,\quad V_j  = pV_{j - 1} -  qV_{j - 2}, \quad  j \ge 2,
\end{equation}
with $U_{ - j}  =\left(pU_{-j + 1} - U_{-j + 2}\right)/q$ and $V_{ - j}  =\left(pV_{-j + 1} - V_{-j + 2}\right)/q$.

\bigskip

Note that, for convenience and since no confusion can arise, we have retained the notation $(W_j)=(W_j(W_0,W_1;p,q))$ for the Horadam sequence and $(U_j)=(U_j(p,q))$ and $(V_j)=(V_j(p,q))$ for the Lucas sequences.

\bigskip

The closed formula for $W_j(W_0,W_1;p,q)$ in the non-degenerate case, $p^2 - 4q > 0$, remains
\begin{equation}\label{BinetH}
W_j = A\tau^j + B\sigma^j, 
\end{equation}
where
\begin{equation}\label{eq.eydt9nt}
A = \frac{{W_1 - W_0 \sigma }}{{\tau  - \sigma }},\quad B = \frac{{W_0\tau  - W_1}}{{\tau  - \sigma }},
\end{equation}
with $\tau$ and $\sigma$ now given by
\begin{equation}\label{eq.egxt2r7}
\tau = \frac{p+\sqrt{p^2 - 4q}}{2}, \qquad \sigma = \frac{p-\sqrt{p^2 - 4q}}{2};
\end{equation}
so that
\begin{equation}\label{eq.p1vf48v}
\tau+\sigma=p,\quad\tau-\sigma=\sqrt{p^2 - 4q}=\Delta,\quad\mbox{and }\tau\sigma=q.
\end{equation}
In particular,
\begin{equation}\label{eq.n6mubfz}
U_j=\frac{\tau^j-\sigma^j}{\tau-\sigma},\quad V_j=\tau^j+\sigma^j.
\end{equation}
Identity~\eqref{eq.dnu2vbt} of the lemma on page~\pageref{eq.dnu2vbt} now reads
\begin{equation}\label{eq.ugio8hh}
A\tau ^j  -  B\sigma ^j  = \frac{{W_{j + 1}  - qW_{j - 1} }}{\Delta }.
\end{equation}
We define Horadam function $w(x)$ by
\begin{equation}\label{eq.c4amqna}
w(x)=A\tau^x + B\sigma^x,\quad x\in\mathbb R,
\end{equation}
where $A$ and $B$ are as defined in~\eqref{eq.eydt9nt} and $\tau$ and $\sigma$ are as given in~\eqref{eq.egxt2r7}.

\bigskip

Equation~\eqref{eq.ydfrjqv} on page~\pageref{eq.ydfrjqv} in the proof of the theorem now reads
\begin{equation}\label{eq.g453iba}
\begin{split}
\left. {\frac{d}{{dx}}w(x)} \right|_{x = j\in\mathbb Z}  &= \left( {A\tau ^j  - B\sigma ^j } \right)\ln \tau  + B\sigma ^j \ln q\\
& = \frac{{W_{j + 1}  - qW_{j - 1} }}{\Delta }\ln \tau  + B\sigma ^j \ln q,\text{by~\eqref{eq.ugio8hh}}.
\end{split}
\end{equation}
If the Horadam parameter, $q$, is \emph{negative}, then we can write~\eqref{eq.g453iba} as
\begin{equation*}
\left. {\frac{d}{{dx}}w(x)} \right|_{x = j\in\mathbb Z} = \frac{{W_{j + 1}  - qW_{j - 1} }}{\Delta }\ln \tau  + B\sigma ^j \ln {(-q)} + B\sigma ^j\ln{(-1)},
\end{equation*}
and hence, taking the imaginary part, we have
\begin{equation}\label{eq.b6o6fl9}
\Im\left( \left. {\frac{d}{{dx}}w(x)} \right|_{x = j \in \mathbb Z} \right) = B_W\sigma^j\pi(2m + 1),
\end{equation}
where $m$ is some integer and
\begin{equation}
B_W = B=\frac{{W_0\tau  - W_1}}{{\tau  - \sigma }}.
\end{equation}
Equation~\eqref{eq.b6o6fl9} is the same as~\eqref{eq.hg8x0pw} on page~\pageref{eq.hg8x0pw} in section~\ref{sec.additional}, except that the values of $\tau$ and $\sigma$ now depend on $p$ and $q$.

\bigskip

The bottomline here is that the method (second component) applies to any general non-degenerate second order sequence (Horadam sequence) $(W_j(W_0,W_1;p,q))$ whose terms are given by
\begin{equation}\label{eq.x1bi81t}
W_j=pW_{j - 1} - qW_{j - 2},
\end{equation}
provided that $q$ is negative.

\bigskip

We now describe how the method (second component) for obtaining new identities from existing ones works for the general second order (Horadam) sequence $\left(W_j(W_0,W_1; p,q)\right)$, $q<0$, whose terms are given in~\eqref{eq.x1bi81t}. The scheme is the following.
\begin{enumerate}
\item \label{step1g1} Let $k$ be a free index in the known identity. Replace each Horadam number, say $W_{h(k,\cdots)}$, with a certain differentiable function of $k$, namely, $w(h(k,\cdots))$, with $k$ now considered a variable.
\item\label{step2g2} By applying the usual rules of calculus, differentiate, with respect to $k$, through the identity obtained in step~\ref{step1g}.
\item\label{step3g3} Simplify the equation obtained in step~\ref{step2g} and take the imaginary part, using also the following prescription:
\begin{quote}
\begin{align}
&\Im \left(w(h(k,\cdots))Y\right)=W_{h(k,\cdots)}\Im Y,\label{eq.pma3e58}\\
&\Im\frac{\partial }{{\partial k}}w(h(k, \cdots )) = B_W\pi\sigma^k=\frac{W_0\tau - W_1}{\Delta}\,\pi\sigma^k\label{eq.mqupsr6};
\end{align}
where $\sigma=(p - \Delta)/2$ and $\Delta=\sqrt {p^2 - 4q} $.
\end{quote}

\end{enumerate}
In particular, for the Lucas sequences, we have
\begin{align}
&\Im \left(u(h(k,\cdots))Y\right)=U_{h(k,\cdots)}\Im Y,\label{eq.eeg6ze7}\\
&\Im\frac{\partial }{{\partial k}}u(h(k, \cdots )) = B_U\pi\sigma^k=-\frac{\pi\sigma^k}\Delta;\label{eq.o6ne7ff}
\end{align}
and
\begin{align}
&\Im \left(v(h(k,\cdots))Y\right)=V_{h(k,\cdots)}\Im Y,\label{eq.g1ginjc}\\
&\Im\frac{\partial }{{\partial k}}v(h(k, \cdots )) = B_V\pi\sigma^k=\pi\sigma^k\label{eq.q19rzh3};
\end{align}
of which the Fibonacci and Lucas relations~\eqref{eq.mbcae9n}--\eqref{eq.b00ifyn} on page~\pageref{eq.mbcae9n} are particular cases.

\bigskip

Note that Proposition~\ref{prop.invariant} on page~\pageref{prop.invariant} on the interchangeability of $\tau$ and $\sigma$, remains valid; and that the method (second component) is applicable provided the expression obtained after substituting Fibonacci, Lucas, gibonacci and Horadam functions $f(x)$, $l(x)$, $u(x)$, $v(x)$, $g(x)$ and $g(x)$ in the given identity holds for all real numbers, as noted in the remark on page~\pageref{sec.b1ek8kn} .

\bigskip

We give an example.

\bigskip

Consider the following identity (\cite[Equation (3.14)]{horadam65}):
\begin{equation*}
U_r W_{k + 1}  - qU_{r - 1} W_k  = W_{k + r}.
\end{equation*}
We write
\begin{equation*}
u(r) w(k + 1)  - qu(r - 1)w(k)  = w(k + r );
\end{equation*}
and differentiate with respect to $r$, obtaining
\begin{equation*}
\frac{d}{{dr}}u(r)\times w(k + 1) - \frac{d}{{dr}}u(r - 1)\times qw(k) = \frac{\partial }{{\partial r}}w(k + r);
\end{equation*}
so that, taking imaginary part, we find
\begin{equation*}
\Im\frac{d}{{dr}}u(r)\times W_{k + 1} - \Im\frac{d}{{dr}}u(r - 1)\times qW_k = \Im\frac{\partial }{{\partial r}}w(k + r);
\end{equation*}
which, using~\eqref{eq.o6ne7ff} and~\eqref{eq.mqupsr6} gives
\begin{equation*}
 - \sigma ^r W_{k + 1}  + q\sigma ^{r - 1} W_k  = \left( {W_0 \tau  - W_1 } \right)\sigma ^{k + r},
\end{equation*}
that is
\begin{equation}\label{eq.vlp82if}
 - \sigma ^r W_{k + 1}  + q\sigma ^{r - 1} W_k  = W_0 q\sigma ^{k + r - 1}  - W_1 \sigma ^{k + r}
\end{equation}
and also
\begin{equation}\label{eq.fyfdjx8}
 - \tau ^r W_{k + 1}  + q\tau ^{r - 1} W_k  = W_0 q\tau ^{k + r - 1}  - W_1 \tau ^{k + r}.
\end{equation}
Combining~\eqref{eq.vlp82if} and~\eqref{eq.fyfdjx8}, using the Binet formula, we have the result stated in Proposition~\ref{prop.adesbgv}.
\begin{proposition}\label{prop.adesbgv}
If $r$ and $k$ are any integers, then
\begin{equation*}
Z_r W_{k + 1} - qZ_{r - 1} W_k = W_1 Z_{k + r} - qW_0 Z_{k + r - 1},\quad q<0;
\end{equation*}
where $W_j=W_j\left(W_0,W_1;p,q\right)$ and $Z_j=Z_j\left(Z_0,Z_1;p,q\right)$ are two Horadam sequences.
\end{proposition}

\hrule

\end{document}